\documentclass[reqno,12pt]{amsart}

\usepackage[left=2.5cm, top=3cm, right=2.5cm, bottom=3cm]{geometry}
\usepackage{amssymb,bbm, bbold}
\usepackage{amsfonts}
\usepackage{amsmath, amscd}
\usepackage{graphicx}
\usepackage{color}
\usepackage{enumerate}

\usepackage{tikz}
\usetikzlibrary{automata}
\usetikzlibrary{arrows,shapes,calc}
\usetikzlibrary{positioning}

\newtheorem{theorem}{Theorem}[section]
\newtheorem{lemma}[theorem]{Lemma}
\newtheorem{proposition}[theorem]{Proposition}
\newtheorem{corollary}[theorem]{Corollary}

\newtheorem{definition}[theorem]{Definition}
\newtheorem{remark}[theorem]{Remark}

\newcommand{\cov}{\mathop{\rm cov}}

\def\W{\mathcal W}

\def\G{\Gamma}

\def\E{{\mathbb E}}

\def\P{{\mathbb P}}

\def\N{{\mathbb N}}

\def\R{{\mathbb R}}

\def\M{{\mathcal M}}

\def\GG{{\mathcal G}}

\def\hg{{\hat g}}

\def\Ups{\Upsilon}

\def\l{\lambda}

\def\tt{\vartheta_{\sigma_2}}

\address{
	\newline
	\newline
	Luiz Renato Fontes
	\newline
	USP - Instituto de Matem\'atica e Estat\'istica.
	\newline  R. do Mat\~ao, 1010 - Butant\~a, S\~ao Paulo - SP, CEP: 05508-090, Brasil.
	\newline
	e-mail:{\rm \texttt{lrfontes@usp.br}}
	\newline
	\newline
	Susana Fr\'ometa
	\newline
	UFRGS - Instituto de Matem\'atica e Estat\'istica.
	\newline
	Campus do Vale, Av. Bento Gonçalves, Porto Alegre - RS, 9500. CEP: 91509-900, Brasil.
	\newline
	e-mail:{\rm \texttt{susana.frometa@ufrgs.br}}
	\newline
	\newline
	Leonel Zuazn\'abar
	\newline
	USP - Instituto de Matem\'atica e Estat\'istica.
	\newline  R. do Mat\~ao, 1010 - Butant\~a, S\~ao Paulo - SP, CEP: 05508-090, Brasil.
	\newline
	e-mail:{\rm \texttt{lzuaznabar@ime.usp.br}}
}


\begin{document}
	
		\date{\today}
		
		\title[Non-Cascading 2-GREM Dynamics at Extreme Time Scales]{ASYMPTOTIC BEHAVIOR OF A LOW TEMPERATURE NON-CASCADING 2-GREM DYNAMICS AT EXTREME TIME SCALES}
		
		\author{ Luiz Renato Fontes, Susana Fr\'ometa and Leonel Zuazn\'abar}
		
		\maketitle
		
		\begin{abstract}
			We derive the scaling limit for the Hierarchical Random Hopping dynamics for the non cascading 2-GREM at low temperatures and time scales where the dynamics is close to equilibrium. The {\em fine tuning} phenomenon 
			plays a role (under certain choices of parameters of the model), yielding three dynamical regimes. 
			In contrast to the {\em cascading} case, the pairs of first and second level energies have fluctuations which scale with the volume, and this leads to a family of time scales where we see the dynamics moving through only a part of the full low temperature energy landscape.
		\end{abstract}	
	
\section{Introduction}\label{sec:intro}

This paper is part of an ongoing effort to understand and describe the long term, large volume behavior of mean field spin glass dynamics at low temperature. It is most directly related to~\cite{fg2018}, where the Hierarchical Random Hopping dynamics (HRHD) for the {\em cascading} 2-GREM was analysed. The present paper takes on the {\em non-cascading} case of the same dynamics.

We may trace the study of such questions/models in the mathematical literature to \cite{abg1},  \cite{abg2}, where the case of the REM was investigated.
The main motivation in many mathematical papers, prompted by the analysis of the phenomenological trap models first appearing in the physics literature (see e.g.~\cite{bd} and~\cite{sn}), is to understand the {\em aging} phenomenon, which takes place away from equilibrium; see also~\cite{adg}, \cite{abc}, \cite{bg}, \cite{gg}, \cite{ag}, \cite{gg}. 

Looking at such dynamics when they are close but not quite in equilibrium, in the {\em ergodic time scale} has also been a matter interest, either in itself,  as a way to understand the transition of the dynamics from an aging behavior to an equilibrium behavior, or yet as an approach to obtain aging results, by first taking a scaling limit at an ergodic time scale, then taking a second limit at a vanishing time scale; see~\cite{fm},~\cite{fl},~\cite{fgg},~\cite{bfggm},~\cite{fg2018}. 

As in~\cite{fg2018}, by an ergodic time scale we mean a time scale where, loosely speaking, and as pointed out above, the dynamics is close to but not quite in equilibrium. 
It is synonymous to the {\em extreme time scale} denomination in the present title, and also that of~\cite{fg2018}.
More precisely, the scaling limit of the dynamics (when both volume and time diverge) under these time scales should be a {\em non trivial, ergodic process}. And under any longer (in leading order) time scale, the dynamics should converge to a product of the (infinite volume) equilibrium distribution (over time) --- in the present case, we should obtain such a longer time scale by multiplying the ergodic time scale by any factor which diverges in the scaling limit. 
In contrast, an {\em aging time scale} is a time scale under which the dynamics exhibits {\em aging} in the scaling limit, i.e., certain of its two-time correlation functions are functions of the {\em quotient} of the two times; this a phenomenon that takes place far from equilibrium, where such correlations would instead be functions of the time {\em difference}, and thus aging time scales should be shorter than ergodic time scales. 

The description of scaling limits at ergodic time scales involve K-processes of one kind or another. This is the case for the cascading 2-GREM of~\cite{fg2018}, and likewise for the present non-cascading case, as will be shown below. K-processes come up also at ergodic time scales of other models, such as trap models on large tori; see~\cite{jlt1},~\cite{jlt2},~\cite{cgl}.  
We briefly discuss these processes at the end of this section.

In order to understand the HRHD at ergodic time scales, we need to have a good grip on the structure of the lowest (2-GREM) energies underlying the model, since the dynamics {\em lives} on the corresponding spin configurations at that time scale. We have information for the full energy profile  from~\cite {bk}, but we need it broken down by hierarchy. Our first result, Theorem~\ref{Thm_1}  in the next subsection, gives an asymptotic description of the lowest 2-GREM energies by hierarchy, detailing the result in~\cite {bk} for the present non-cascading case. It says that in that case, while the full energies are at distances of order 1 apart from each other, as established in~\cite{bk}, their components by hierarchy are much farther apart, in a volume dependent way. This is in stark contrast with the cascading case, where both hierarchical components of the lowest total energies behave in this respect  similarly as their total. This makes for a different behavior of the noncascading dynamics at ergodic time scales in two related aspects: 
\begin{itemize}
	\item  as in the cascading case, a {\em fine tuning} temperature phenomenon takes place in the noncascading dynamics as well (in a certain region of the parameter space), {\em but}
	while in the cascading case the fine tuning range of temperatures is of order $N^{-1}$, in the noncascading case it is of order $N^{-1/2}$,
	where $N$ represents the volume;
	\item more strikingly, at and above fine tuning (low) temperatures, the ergodic time scale admits a range of values indexed by a real parameter $L$,
	which in a certain way {\em selects} which lowest energy configurations are visited by the dynamics at a given such time scale.
\end{itemize}

Differently from~\cite{fg2018}, we refrain in this paper to discuss aging behavior. One reason is concision, but perhaps the main reason is that we expect the behavior on those time scales to be no different than for the REM; there should be no surprises, in contrast to what we find at ergodic time scales. 

We present a detailed discussion of the above points in the following subsections.

The HRHD may be seen as Markov jump process on the hypercube $\{-1,1\}^N$ in a random environment. 
We start its description by the environment in the next subsection. A definition of the dynamics, with a discussion of its relevant characteristics, followed by our main (dynamical) results, occupy the remaining subsections of this introduction.

\subsection{The environment.}

The Generalized Random Energy Model (GREM) was introduced in~\cite{derrida1985generalization} as a hierarchical spin glass in equilibrium. We consider here the case with two hierarchies (the 2-GREM), which we specify next.

Given a natural number $N$ and $p\in(0,1)$ let us define $N_1 :=\left[pN\right]$ and $N_2:=N-N_1$. Consider, in $\mathcal{V}_N:=\{-1,1\}^N$, a vector $\sigma=\sigma_1\sigma_2$, where $\sigma_1\in\mathcal{V}_{N_1}:=\{-1,1\}^{N_1}$ and $\sigma_2\in\mathcal{V}_{N_2}:=\{-1,1\}^{N_2}$. For a given $a\in(0,1)$, and for each $N$, let us define the following Gaussian random variable
\begin{equation}\label{gaussians}
	\Xi_{\sigma}:=\sqrt{a}\Xi_{\sigma_1}^{(1)}+\sqrt{1-a}\Xi_{\sigma_1\sigma_2}^{(2)},
\end{equation}
where $\Xi=\{\Xi_{\sigma_1}^{(1)}, \Xi_{\sigma_1\sigma_2}^{(2)}:\sigma\in\mathcal{V}_N \}$ is a family of independent standard Gaussian random variables. This represents our \textit{random environment}. The variables on the family $\Xi$ are correlated in the following way: $\cov(\Xi_\sigma,\Xi_\tau)=0$ if $\sigma_1\neq\tau_1$ and $\cov(\Xi_\sigma,\Xi_\tau)=a$ if $\sigma_1=\tau_1$ and $\sigma_2\neq\tau_2$.  

The associated Gibbs measure at inverse temperature $\beta>0$ is the (random) measure $G_{\beta,N}(\sigma)=e^{\beta\sqrt{N}\Xi_\sigma}/Z_{\beta,N}$ where $Z_{\beta,N}$ is a normalization. The Hamiltonian or \textit{energy} of $\sigma$ would be, as usual, $H_N(\sigma)=-\sqrt{N}\Xi_\sigma$. We refer to the minima of $H_N(\cdot)$ as \textit{low energy} or \textit{low-lying} configurations. At low temperature, the Gibbs measure is essentially concentrated on the low-lying configurations, and at high temperature, the Gibbs measure is concentrated in a growing number of energy levels which get denser as $N\to\infty$.

There are two scenarios that may be distinguished: the \textit{cascading case} and the \textit{non-cascading case}. The cascading occurs when $a>p$, and the low-lying energies are{
	achieved} by adding up the low-lying energies of the two levels. The non-cascading case occurs when $a\leq p$ and it turns out that the correlations are{
	too} weak to have an impact on the extremes and the system ``collapses'' to a REM, see \cite{bk}. Also, in this case, the extremal configurations must differ in the first index, this phenomenon justifies the non-cascading denomination for our system.

As we mentioned before, we will consider the non-cascading 2-GREM evolving under the Randon Hopping Dynamics at ergodic time scales. Before introducing the dynamics we will go into the behavior of the environment.

Our main results involve the point process that we define below.

\begin{definition}\label{PPP}
	Let us call $\mathcal{P}=\left\{\xi_i,\, i\geq 1 \right\}$ the Poisson point process{
		on $\mathbb R$} with intensity measure  $Ke^{-x}dx$, with $K>0$, such that $\xi_i>\xi_{i+1}$ for all $i\geq 1$. 
	We note that this process has a finite maximum. 
\end{definition}

Concerning the environment, we need to establish a limit result for the first coordinate of the low-lying configurations. We mentioned above that in the non-cascading case the extremal values behaves like in the REM. Indeed, it was proved in Theorem 1.1 of \cite {bk} that, in the case $a\leq p$, the point process 
\begin{equation}\label{BK}
	\mathcal{P}_N:=\sum_{\sigma\in\mathcal{V}_{N_1}\times\mathcal{V}_{N_2}}\mathbb{1}_{u_N^{-1}(\Xi_{\sigma})}\to\mathcal{P}\text{~~in distribution, as~} N\to\infty,
\end{equation}
where $\mathcal{P}$ is the Poisson point process of Definition \ref{PPP}, with $K=1$ when $a<p$ and $K=\frac{1}{2}$ when $a=p$, and the function $u_N$ is the scaling function for the maximum of $2^N$ i.i.d. standard Gaussians, defined as 
\begin{equation}\label{def2}
	u_N(x)=\frac{x}{\beta_*\sqrt{N}}+\beta_*\sqrt{N}-\frac{\log N+\kappa}{2\beta_*\sqrt{N}},
\end{equation}
for $\kappa =\log\log 2+\log 4\pi$ and $\beta_*=\sqrt{2\log 2}$.

The constant $\beta_*$ above also coincides with the singularity point of the free energy function associated to the Gibbs measure. The low-temperature regime occurs when $\beta>\beta_*$. As we mentioned before, in this temperature region, the Gibbs measure becomes,{
	in the large volume limit}, fully concentrated on the set of low-lying configurations,{
	corresponding to the (suitably exponentiated and normalized) points of $\mathcal{P}$ --- see~(\ref{convergence_exp}) below}.

The convergence in \eqref{BK}, however, does not give us any information about the composition of the low-lying configurations when looking at the first and second level separately. The following result states the convergence of the first level low-lying configurations.

\begin{theorem}[Environment behavior]\label{Thm_1}  Let us relabel our{
		indices} $\sigma=(\sigma_1,\sigma_2)$ as $\sigma(i)=(\sigma_1(i),\sigma_2(i))$ in order to have $\Xi_{\sigma(1)}>\Xi_{\sigma(2)}>\cdots >\Xi_{\sigma(2^N)}$\footnote{This is well defined a.s.}. For $a\leq p$ and $\mathcal{P}=\{\xi_i,\,i\geq 1\}$  the Poisson point process in Definition \ref{PPP}, 
	we have that for every $k\geq1$
	\begin{align*}
		\left( u_N^{-1}(\Xi_{\sigma(1)}), \Xi^{(1)}_{\sigma_1(1)}-\sqrt{aN}\beta_*;\dots;u_N^{-1}(\Xi_{\sigma(k)}), \Xi^{(1)}_{\sigma_1(k)}-\sqrt{aN}\beta_*\right)
	\end{align*}
	converges in distribution to $\left(\xi_1,W_1;\dots;\xi_k,W_k\right)$, where, for $a<p$, we have $K=1$ and $W_1,\cdots,W_k$ are independent Gaussian variables with mean zero and variance $1-a$; which are also independent of the process $\mathcal{P}$, and for $a=p$, we have $K=\frac{1}{2}$ and the variables $W_1,\cdots,W_k$ are now conditioned on being negative.
\end{theorem}

\begin{remark}
	Theorem \ref{Thm_1} can be written considering instead the second level of the low-lying configurations, obtaining an analogous result 
	in which the{
		centered} Gaussian variables in the limit are also independent, but with variance $a$. 
\end{remark}

{
	\begin{remark}
		It follows from Theorem \ref{Thm_1} that in the limit $\{\sigma_1(i) : i\geq1\}$ are all distinct. In other words, for each 
		$\sigma_1(i)$, we have a single $\sigma_2(i)$, and we thus say that 	$\sigma_1(i)$ are $\sigma_2(i)$ are {\em matched}. 
		This is in contrast with the cascading phase, where to each 
		$\Xi^{(1)}_{\sigma_1(i)}$, there corresponds (in the limit, infinitely) many ${\Xi^{(2)}_{\sigma_2(\cdot)}}$'s.
	\end{remark}
	
	\begin{remark}
		A similar issue as the one treated in Theorem \ref{Thm_1}  arises in the analysis of  the maximum of two-speed branching Brownian motion, which is a model that is quite similar, and indeed related, to the 2-GREM; see \cite{bh}. The proof of  Theorem 1.2 of the latter reference goes through an estimation of two components of maxima, similarly as in our proof of our Theorem \ref{Thm_1}, and a range of square root of the leading order of the contributions to the maxima, around such leading order, comes up as well; see (outline of) the proof of Theorem 1.2 of \cite{bh}. For the analysis of the HRHD of the 2-GREM we need however precise location of such contributions, beyond the order of magnitude of such range; this is provided by our Theorem \ref{Thm_1} above, and, as far as we can see, not in \cite{bh}, or elsewhere.
	\end{remark}
	
	For later use, let us denote 
	\begin{equation}\label{gamma}
		\gamma^N(\sigma_1\sigma_2)=e^{\frac{\beta}{\beta_*}u_N^{-1}(\Xi_{\sigma_1\sigma_2})}\text{~~~~~~~~~and~~~~~~~~~~}\gamma_i=e^{\frac{\beta}{\beta_*}\xi_i}.
	\end{equation}
	When $\sigma_1\sigma_2=\sigma(i)$, we denote $\gamma^N_i=\gamma^N(\sigma_1\sigma_2)$. 
	
	As a consequence of \eqref{BK} we have 
	
	\begin{equation}\label{convergence_exp}
		\lim_{N\to\infty}\Big\{\sum_{\sigma_2\in\mathcal{V}_{N_2}}\gamma^N(\sigma_1(i)\sigma_2),\,i\geq1\Big\}
		=\G:=\big\{\gamma_i,\,i\geq1\big\},\text{~~~~~~in distribution.}
	\end{equation}
	
	We note that $\G$ is a Poisson point process in $[0,\infty)$ with intensity $\frac\alpha{x^{1+\alpha}}\,dx$,
	where $\alpha=\frac{\beta_*}\beta$; for $\beta>\beta_*$, the sum over the points of $\G$ is a.s.~finite.

	\subsection{{Dynamics.}}

	We consider the Hierarchical Random Hopping dynamics (HRHD)}, which is a Markov jump process $\{\sigma^N(t),\,t>0\}$ that evolves in $\mathcal{V}_N$ with transition rates given by 
\begin{equation}\label{dynamics}
	w_N(\sigma,\sigma')=\frac{e^{-\beta\sqrt{N}\Xi_\sigma}}{N}\mathbb{1}_{\sigma \stackrel{1}{\sim} \sigma'}+\frac{e^{-\beta\sqrt{(1-a)N}\Xi^{(2)}_\sigma}}{N}\mathbb{1}_{\sigma \stackrel{2}{\sim} \sigma'}
\end{equation}
and $w_N(\sigma,\sigma')=0$ else, where, for $i=1,2$, we say that $\sigma \stackrel{i}{\sim} \sigma'$ iff $\sigma\sim \sigma'$ and $\sigma_i \sim \sigma'_i$. Here, $\sigma\sim \sigma'$ indicates that $\sigma$ and $\sigma'$ differs in exactly one coordinate. We say, in this context, that $\sigma$ and $\sigma'$ are nearest neighbors in $\mathcal{V}_N$. 

The	low-temperature 2-GREM in the non-cascading case, as well as in the cascading case, exhibits,
depending on the values of $p$ and $a$, a dynamical phase transition in the HRHD at ergodic time scales that
may be described as a \textit{fine tuning} phenomenon.  

\subsubsection{Fine tuning; heuristics}\label{heuristics}
There are two competing factors governing the behavior of the HRHD {\em at ergodic time scales}. One factor is $\#_2$, the number of jumps until $\sigma^N$ finds a second level low-lying configuration;
this factor is of order $2^{N_2}$. The other factor is $\#_1$, the number of visits 
by $\sigma^N$ to a first level low-lying configuration $\sigma_1$ before leaving it. 
This is a geometric random variable with mean $1+\frac{N_2}{N_1}e^{\beta\sqrt{aN}\Xi_{\sigma_1}^{(1)}}$, which, by Theorem \ref{Thm_1}, is of order $e^{\beta\beta_*aN+\beta\sqrt{aN}W}$, where $W$ is a centered Gaussian random variable with variance $1-a$. This differs with the case $a>p$, studied in \cite{fg2018}, where the first level of a low-lying configuration is exponentially large in $N$ (namely, {
	of  order $e^{\beta\beta_*\sqrt{pa}N}$), but with a deterministic inverse square root correction, instead of the present correction, which is random and stretched exponential.

	The Gaussian random variable $W$ that appears in our case will play a crucial role in the behavior of the process $\sigma^N$ at low temperature at or above fine tuning temperatures. The first effect is in the definition of the ergodic time scale in those regimes, which will incorporate a threshold $L$ in the exponential root scale. This will effectively select low-lying configurations whose $W$ values exceed the threshold, and those configurations are the only low-lying ones visited by the dynamics in that time scale. We thus get a family of ergodic time scales, indexed by a real parameter $L$, the scaling limit of the dynamics under which yields an ergodic infinite volume dynamics supported on the selected low-lying configurations for each $L$. 
	
	The analysis in the paragraph before last leads us to define a family of volume dependent critical parameters  
	\begin{equation}\label{beta_{FT}}
		\beta_{FT}:=\frac{(1-p)}{2a}\,\beta_*-\frac\theta{\sqrt N},\,\theta\in\R,\mbox{ and }\bar\beta_{FT}:=\frac{(1-p)}{2a}\,\beta_*.
	\end{equation}
	
	Let us assume henceforth that $\bar\beta_{FT}>\beta_*$;
	we will consider three different temperature regions: when $\beta$ is larger than $\bar\beta_{FT}$, equal to $\beta_{FT}$ and smaller than $\bar\beta_{FT}$.

	\begin{enumerate}
		\item [(i)] At relatively low temperature, when $\beta>\bar\beta_{FT}$, then $\#_1\gg\#_2$, and while staying at a first level low-lying configuration, the process has time to reach the matching second level low-lying many times until it changes the first component. Here we have a similar situation as the one described in Theorem 2.7 of \cite{fg2018}. 
		
		\item[(ii)] At fine tuning $\beta=\beta_{FT}$, as pointed out above, the ergodic time scale has a threshold parameter $L=L_\theta$; then, we have $\#_1\gg\#_2$ when $W>L$, and $\#_1\ll\#_2$ when $W<L$. Hence for some low-lying configurations the process has time to reach the second level maximum
		(many times) during each visit to the first level of such configurations, and there are other ones for which the first component changes before the process hits the second level ground configurations. So, in this case the HRHD behaves like for $\beta>\bar\beta_{FT}$, except that it does not visit part of the set of low-lying configurations
		(those for which $W<L$), but rather spends virtually all of the time in the complementary set of low-lying configurations.
		%
		%
		\item[(iii)] At relatively high temperature, when $\beta_*<\beta <\bar\beta_ {FT} $, we have that $\#_1\ll\#_2$; then, after leaving a low-lying configuration, $ \sigma^N $ will visit many first level low-lying configurations without finding the matching second level low-lying configuration.  We again have the threshold-$L$ time scale,and again low-lying configurations with $W<L$ are not visited. The difference here with respect to the fine tuning regime is that each low-lying configuration that gets visited, is visited many times, in such a way that every time a configuration with a certain $W'$ is left, before it returns, the ones with values of $W$ larger than $W'$ are visited many times; 
		however, there is a compensating factor in that the dynamics spends much more time in each visit to configurations with smaller $W$'s, so that the total time spent 
		on visits to low-lying configurations with  $W$'s larger than $W'$ between visits to the low-lying configuration corresponding to  $W'$ is comparable to that of a single visit to the latter configuration.	
		This provides a mixing mechanism for the visits to the low-lying configurations which results in the infinite volume dynamics being in a sort of {\em partial equilibrium}, or equilibrium in only a part of the configuration space 
		(namely, those configurations with $W>L$).
	\end{enumerate}

	More details and precise statements are given below.

	\begin{remark}
		In the case where $\bar\beta_{FT}\leq\beta_*$, the fine tuning phenomenon takes place at high/critical temperature, and since we are looking only at subcritical temperatures, we only have the behavior described in (i), in that case.
	\end{remark}

	\subsection{Scaling limit results for the dynamics.}\label{section.main}
	
	In this section we will state our limit results for the dynamics according to three different conditions on the temperature parameter $\beta$: when $\beta$ is larger, equal or smaller than the parameter $\beta_{FT}$, defined in \eqref{beta_{FT}}. We will use the convergence stated in Theorem \ref{Thm_1}, which is weak in the original environment; however, via Skorohod's Theorem, we are able to make it a strong convergence by going to another, suitable probability space. From now on, we will assume that we are in the probability space where the strong convergence takes place, and omit further reference to it.
	
	For each $N$, we relabel the indices $\sigma=(\sigma_1,\sigma_2)$ as $\{ \sigma(1),\dots,\sigma(2^N) \}$ such that $\Xi_{\sigma(1)}>\Xi_{\sigma(2)}>\cdots>\Xi_{\sigma(2^N)}$ (such notation was already used in Theorem \ref{Thm_1}). Let us define the function $\phi^N:\mathcal{V}_N\to\mathbb{Z}_+$ as  
	\begin{equation*}
		\phi(\sigma)=\phi(\sigma_1\sigma_2):=\min\{i=1,\dots, 2^N:\,\sigma_1(i)=\sigma_1\}\, ,
	\end{equation*}
	and the process $X^N$ as
	\begin{equation}\label{def1}
		X^N(t):=\phi(\sigma^N(t))\, , \text{ for all } t\ge 0\, .
	\end{equation}
	
	Our scaling limit results will be stated for the process $X^N$ for the convenience of working with a state space which naturally extends to the set of natural numbers, which will be the state space of the limiting processes. Recall the assumption made right below~(\ref{beta_{FT}}). We will assume a uniform initial distribution for $\sigma^N(t)$, for the convenience that the
	uniform distribution on $\mathcal{V}_{N_i}$ is invariant for the random walk on the hypercube $\mathcal{V}_{N_i}$, $i=1,2$.
	It is a simple, if cumbersome matter to extend our results to most other initial distributions.
	
	{
		Let $\mathcal{P}=\{\xi_i:i\geq 1\}$ and $\W:=\{W_i,i\geq 1\}$ be as in Theorem \ref{Thm_1}. 
		Given $L\in \mathbb{R}$ when $a<p$, and $L<0$ when $a=p$, let us define
		\begin{equation}\label{c_N}
			c_N=c^L_N=c^L_N(\beta):= e^{\beta\left(\beta_*N- \frac{\log N +\kappa}{2\beta_*}\right)}e^{-\beta(\beta_*aN+\sqrt{aN}L)},
		\end{equation}
		where $\kappa$ is the same as in \eqref{def2}.
		
		Recall~(\ref{convergence_exp}) and set $\N_L=\{i\geq1 : W_i>L\}$, $\G_L=\{\gamma_i : i\in\N_L\}$, $\pi^L_\ell=\frac{\gamma_\ell}{\sum_{i\in\N_L}\gamma_i}$, $\ell\in\N_L$.
		
		\begin{theorem}[Above fine tuning temperatures]\label{Thm_above_FT} 
			For $\beta_*<\beta<\bar\beta_{FT}$, $t>0$ and $\ell\in\N_L$, we have that
			\begin{equation*}
				\lim_{N\to\infty}\frac{1}{c_Nt}\int_{0}^{c_Nt}\mathbb{1}_{\{X^N(s)=\ell\}}ds = \pi^L_\ell\, . 
			\end{equation*}
			The convergence above holds in probability.	
		\end{theorem}	
		
		\begin{remark}
			Theorem~\ref{Thm_above_FT} follows from an analysis based on the heuristics described in (iii) of Subsubsection~\ref{heuristics}. It suggests that the single-time (and also the finite dimensional) distributions of $\{X^N(c_Nt),\,t\geq0\}$ converge to (products of) $\pi^L_\cdot$; however, we did not find a way to show that; the main difficulty issues from the increasingly fractal nature of the successive visits of the process to the relevant low-lying configurations as $N\to\infty$: as indicated in the above mentioned heuristics, we have different increasingly  incomparable single time scales for different such configurations in each visit; furthermore, these single time scales are increasingly infinitesimal with respect to the overall time scale, but add up to it when summed over all visits, and successive visits to a given such configuration are interspersed with an increasing number of visits to other such configurations. 
		\end{remark}
		
		\begin{theorem}[At fine tuning temperature]\label{Thm_at_FT} 
			For $\beta=\beta_{FT}$, with $\theta=\frac{1-p}{2a^{3/2}}L$ in~\eqref{beta_{FT}}, we have that $\{X^N(c_Nt),\,t\geq0\}$ converges in distribution as $N\to\infty$ to a K-process on $\bar\N_L=\N_L\cup\{\infty\}$ with parameter set $\G_L$, starting from $\infty$.
		\end{theorem}
		
		We have a longer ergodic time scale for the next result. Let 
		\begin{equation}\label{bar_c_N}
			\bar c_N=2^{-N_2}e^{\beta(\beta_*N - \frac{\log N +\kappa}{2\beta_*})}.
		\end{equation}
		
		\begin{theorem}[Below fine tuning temperatures]\label{Thm_below_FT} 
			For $\beta>\bar\beta_{FT}$, we have that $\{X^N(\bar c_Nt),\,t\geq0\}$ converges in distribution as $N\to\infty$ to a K-process on $\bar\N=\N\cup\{\infty\}$ with parameter set $\G$, starting from $\infty$.
		\end{theorem}	
		
		The convergence of Theorems~\ref{Thm_at_FT}  and \ref{Thm_below_FT} is on Skorohod space with the $J_1$ topology.
		{
			\begin{remark}
				Our choice of representation for the state of the dynamics at each time, involving the first level part of the spin configuration only, as prescribed in~\eqref{def1}, allows for the convergence in the $J_1$-Skorohod space. 
				A reasonable alternative choice would be to include the second level part of the spin configuration, but then we 
				would lose the $J_1$-convergence in Theorems~\ref{Thm_at_FT} and~\ref{Thm_below_FT} (since in both cases we have 
				increasingly many jumps in and out of the second level part of low-lying configurations, while the second level part
				rests). 
				Since in the non cascading regime we have, for each low-lying configuration, a single second level part for
				each first level part, the choice of representation between both possibilities considered is more of a technical nature,
				pertaining rather to the mode of convergence than to the limiting dynamics.
			\end{remark}

			%
			
			\begin{remark}
				As pointed out at the beginning of this subsection, Theorems~\ref{Thm_above_FT},~\ref{Thm_at_FT} and~\ref{Thm_below_FT} are stated for a special version of the environment, given by Skorohod representation, which converges almost surely as $N\to\infty$. Convergence results for the original environment follow immediately for the {\em integrated dynamics}, that is, the distribution of the dynamics, before and after the limit, averaged over the respective environment.
			\end{remark}
			
			\begin{remark}
				A comparison to the corresponding results of~\cite{fg2018} is in order. Theorem~\ref{Thm_below_FT} is essentially the same as
				Theorem 2.7 of~\cite{fg2018} if one considers only the first level motion in the latter theorem, which is in a sense the only relevant one (since, as will be argued in detail below, the dynamics spends virtually all the time visiting the low-lying configurations, and, as pointed out above, the second level part of the each low-lying configuration is a function of the
				first level part of that configuration). 
				For Theorems~\ref{Thm_above_FT} and~\ref{Thm_at_FT}, the different structure of the minima of the energies with respect to the cascading case of~\cite{fg2018}, revealed by Theorem~\ref{Thm_1}, is felt, both in the (family of) scales in the ergodic time regime, as well as in the fact that in those time scales we only see part of the energy landscape (and not the full one, as 
				in the corresponding results of~\cite{fg2018}). Other than that (but this is of course a major point), Theorems~\ref{Thm_above_FT} and~\ref{Thm_at_FT} describe similar behavior of the non cascading dynamics, as Theorems~2.4 and~2.5 of~\cite{fg2018} do for the cascading case, if one looks only at the first level motion of the latter dynamics; we notice however that, while for the cascading case, in the time scale of Theorem~2.4 of~\cite{fg2018}, the first level is in full equilibrium, this is not the case described by the present Theorem~\ref{Thm_above_FT}, which, as pointed out above, could be better characterized as partial equilibrium.
			\end{remark}
			
			
			
		}
		
		\subsection{K-processes}		
		We give a brief description of a process entering two of our main results. By a K-process in this paper we loosely mean a Markov process on an infinite  subset $\mathcal N$ of $\N$ whose single visits to a given $x\in\mathcal N$ lasts an exponential time of mean $\gamma_x$,  where $\gamma_x>0$, $x\in\mathcal N$, are  parameters of the process, satisfying moreover that $\sum_{x\in{\mathcal N}}\gamma_x<\infty$. The transition from $x$ is {\em uniform} in a certain sense, which can be made precise given the latter summability of the parameters.
		See~\cite{fm}, Section 3 (where ${\mathcal N}=\N$, and there is an extra parameter  $c$, which for our purposes here should be taken as $0$), for a detailed definition and properties of such a (version of this) process. In order to have regular (i.e., c\`adl\`ag) trajectories, we compactify $\mathcal N$, and the resulting process lives in ${\mathcal N}\cup\{\infty\}$. The set of times the process spends at $\infty$ is Cantor-like/perfect and with vanishing length. An interpretation of the sites of $\mathcal N$ in contrast to $\infty$ in the context of scaling limits of low-temperature spin dynamics is that $x\in\mathcal N$ represents singly a general low-lying (selected) configuration, while $\infty$ represents the remaining (higher energy) configurations lumped together.

		Variants of this process come up in similar situations involving a range of  dynamics which exhibit trapping, as pointed out above, under names such as {\em weighted K-processes} (as in~\cite{fp1}; see also~\cite{bfggm, jlt1, jlt2}); (multi-dimensional) {\em K-processes on a tree} (as in \cite{fgg, fp2}), or {\em spatial K-processes} (as in~\cite{cgl}). The only version appearing in the present paper, however, is the one discussed in the previous paragraph.

		\subsection{Organization of the paper.} 
		
		We devote the next section for the proof of Theorem~\ref{Thm_1}, and the following three sections to 
		Theorems~\ref{Thm_above_FT},~\ref{Thm_at_FT}, and~\ref{Thm_below_FT}, respectively.

		\section{Proof of Theorem \ref{Thm_1}}
		\setcounter{equation}{0}

		We will prove first the case $a<p$. The case $a=p$ will follow the lines of the previous case, as we will see later.

		Let us consider, for $\delta>0$ and $j\in\mathbb{Z}$, the intervals
		\begin{equation}\label{def3}
			I_N^j:=\left[\sqrt{aN}\beta_* + \frac{j}{N^{\frac{1}{2}+\delta}},\sqrt{aN}\beta_* + \frac{j+1}{N^{\frac{1}{2}+\delta}}\right].
		\end{equation}	
		For each $j$, let us call $ \mathcal{C}(I_N^j) := \#\{\sigma_1\in\mathcal{V}_{N_1}: \Xi_{\sigma_1}^{(1)}\in I_N^j\}$ the cardinality of the set of index $\sigma_1$  in $\mathcal{V}_{N_1}$ such that $\Xi^{(1)}_{\sigma_1}$ belongs to the interval $I_N^j$. Let us consider
		\begin{equation}\label{defmax}
			M_N^j:=\max_{\sigma\in\mathcal{V}_{N}:\Xi_{\sigma_1}^{(1)}\in I_N^j}\Xi_\sigma.
		\end{equation}
		Note that $\displaystyle \max_{\sigma\in\mathcal{V}_N}\Xi_{\sigma}=\max_{j\in\mathbb{Z}} M_N^j$. 
		Let us consider
		\begin{equation}\label{def6}
			Y_N^j:=\sqrt{a}\left(\sqrt{aN}\beta_*+\frac{j}{N^{\frac{1}{2}+\delta}}\right)+\sqrt{1-a}\left(\max_{\sigma\in\mathcal{V}_{N}:\Xi_{\sigma_1}^{(1)}\in I_N^j}\Xi_{\sigma_1\sigma_2}^{(2)}\right).
		\end{equation}
		Later,  in Lemma \ref{MN_approx_YN}, we will prove that $M_N^j$ and $Y_N^j$ are{
			{\em close}}. Still, in $Y_N^j$ we are taking the maximum over a set of random size. In order to overcome this difficulty, we will replace $\mathcal{C}(I_N^j)$ by $\mathbb{E}\left[\mathcal{C}(I_N^j)\right]$ and, in Lemma \ref{second moment}, we establish a result that enables  this replacement.

		For each $j$, let us define the variables 
		\begin{equation}\label{def4}
			W_N^j:=\max_{k=1,\dots,2^{N_2}\mathbb{E}[\mathcal{C}(I_N^j)]}\widetilde{\Xi}_{j,k},
		\end{equation}
		where $\{\widetilde{\Xi}^N_{j,k}: j\geq 1, k\geq 1\}$ is a family of i.i.d. standard Gaussian variables, and
		\begin{equation}\label{def5}
			Z_N^j:= \sqrt{a}\left(\sqrt{aN}\beta_*+\frac{j}{N ^{\frac{1}{2}+\delta}}\right)+\sqrt{1-a}W_N^j.
		\end{equation}
		In order to prove Theorem \ref{Thm_1}, we will first prove some convergence results for the variables $Z_N^j$,
		as an approximation to $Y_N^j$. Later we prove that this approximation is {\em good} and 
		Theorem \ref{Thm_1} follows.
		
		\begin{lemma}\label{lemma2} 
			Let $B$ be a nonempty open subset of $\mathbb{R}$ and $\epsilon\in(0,\frac{1}{2})$; then, we have that
			\begin{align*}
				&\lim_{N\to\infty}\mathbb{P}\left[\max_{j:\frac{j}{N^{\frac{1}{2}+\delta}}\in B\cap[-N^{\epsilon},N^{\epsilon}]}u_N^{-1}(Z_N^j)\leq x\right]=e^{-\Phi_Be^{-x}},
			\end{align*}
			where $\displaystyle \Phi_B=\int_B\frac{1}{\sqrt{2\pi(1-a)}}e^{-\frac{x^2}{2(1-a)}}dx$.
		\end{lemma}
		\begin{remark}\label{const}
			The constant $\Phi_B$ in Lemma \ref{lemma2} hints at the appearance of the Gaussian random variables in the statement of Theorem \ref{Thm_1}. 
		\end{remark}	
		\begin{remark}\label{ext}
			It readily follows from Lemma~\ref{lemma2} that its statement extends to other $B$'s, like closed subsets of $\mathbb{R}$.
		\end{remark}
		\begin{proof}
			For $x$ fixed, by independence, we have 
			\begin{align*}
				\mathbb{P}\left[\max_{j:\frac{j}{N^{\frac{1}{2}+\delta}}\in B\cap [-N^\epsilon,N^\epsilon]}Z_N^j\leq u_N(x)\right]&=\prod_{j:\frac{j}{N^{\frac{1}{2}+\delta}}\in B\cap [-N^\epsilon,N^\epsilon]}\mathbb{P}\left[W_N^j\leq y_N^j\right]\\
				&=\prod_{j:\frac{j}{N^{\frac{1}{2}+\delta}}\in B\cap [-N^\epsilon,N^\epsilon]}\mathbb{P}\left[X\leq y_N^j\right]^{2^{N_2}\mathbb{E}[\mathcal{C}(I_N^j)]}.
			\end{align*}
			where $\displaystyle y_N^j=y_N^j(x)=\frac{\frac{x}{\beta_*}+(1-a)\beta_*N-\frac{\log N+\kappa}{2\beta_*}-\frac{\sqrt{a}j}{N^\delta}}{\sqrt{(1-a)N}}$ and $X$ is a standard Gaussian random variable.
			
			\begin{remark}\label{yn}
				We note that $\min_{j:\frac{j}{N^{\frac{1}{2}+\delta}}\in B\cap[-N^{\epsilon},N^{\epsilon}]}y_N^j\to\infty$ as $N\to\infty$.
			\end{remark}
			
			Applying logarithms we can see that it suffices to show that
			\begin{equation}\label{eq1}
				\lim_{N\to\infty}\sum_{j:\frac{j}{N^{\frac{1}{2}+\delta}}\in B\cap[-N^{\epsilon},N^{\epsilon}]}2^{N_2}\mathbb{E}[\mathcal{C}(I_N^j)]\log\left(1-\mathbb{P}\left[X> y_N^j\right]\right)=-\Phi_Be^{-x}.
			\end{equation}
			By Taylor's Theorem, we have that
			\begin{equation}\label{taylor}
				\log\left(1-\mathbb{P}\left[X>y_N^j\right]\right)=
				-\mathbb{P}\left[X>y_N^j\right]-\frac{\mathbb{P}\left[X>y_N^j\right]^2}{2(1-\theta_N^j)}, 
				\text{ where } 0<\theta_N^j<\mathbb{P}\left[X>y_N^j\right]<1.
			\end{equation}
			Then, since, by Remark~\ref{yn}, $\max_{j:\frac{j}{N^{\frac{1}{2}+\delta}}\in B\cap[-N^{\epsilon},N^{\epsilon}]}\mathbb{P}\left[X> y_N^j\right]$ vanishes as
			$N\to\infty$, it suffices to prove  
			\begin{equation}\label{eq9}
				\lim_{N\to\infty}\sum_{j:\frac{j}{N^{\frac{1}{2}+\delta}}\in B\cap[-N^{\epsilon},N^{\epsilon}]}2^{N_2}\mathbb{E}[\mathcal{C}(I_N^j)]\mathbb{P}\left[X> y_N^j\right]=\Phi_Be^{-x}.
			\end{equation}
			
			A simple computation shows that 
			\begin{align}
				2^{N_2}\mathbb{E}[\mathcal{C}(I_N^j)]&=2^N\mathbb{P}\left[X\in I_N^j\right]=\frac{2^N}{\sqrt{2\pi}}\int_{I_N^j}e^{-\frac{x^2}{2}}dx=\frac{2^N\eta_N^j}{\sqrt{2\pi}N^{\frac{1}{2}+\delta}}\exp\left\{-\frac{\left(\sqrt{aN}\beta_*+\frac{j}{N^{\frac{1}{2}+\delta}}\right)^2}{2}\right\},
			\end{align}
			where $\eta_N^j\to 1$ as $N\to\infty$ for all $|j|\leq N^{\frac{1}{2}+\delta+\epsilon}$ and $\epsilon<\frac{1}{2}$. Then, in order to prove \eqref{eq9}, it suffices to prove
			\begin{equation}\label{eq12}
				\lim_{N\to\infty}\sum_{j:\frac{j}{N^{\frac{1}{2}+\delta}}\in B\cap[-N^{\epsilon},N^{\epsilon}]}\frac{2^N}{\sqrt{2\pi}N^{\frac{1}{2}+\delta}}\exp\left\{-\frac{\left(\sqrt{aN}\beta_*+\frac{j}{N^{\frac{1}{2}+\delta}}\right)^2}{2}\right\}\mathbb{P}\left[X> y_N^j\right]=\Phi_Be^{-x}.
			\end{equation}

			We recall the following standard tail estimates on the tail of the standard (unnormalized) Gaussian distribution 
			\begin{equation}\label{bds}
				\frac x{1+x^2} e^{-\frac{x^2}{2}}\leq\int_x^\infty e^{-\frac{t^2}{2}}dt\leq\frac1xe^{-\frac{x^2}{2}},\,x>0.
			\end{equation}
			Then, again by Remark~\ref{yn}, and using~(\ref{bds}), we have that in order to prove \eqref{eq12}, it suffices to show that
			\begin{equation}\label{eq13}
				\lim_{N\to\infty}\sum_{j:\frac{j}{N^{\frac{1}{2}+\delta}}\in B\cap[-N^{\epsilon},N^{\epsilon}]}\frac{2^N}{\sqrt{2\pi}N^{\frac{1}{2}+\delta}}\exp\left\{-\frac{\left(\sqrt{aN}\beta_*+\frac{j}{N^{\frac{1}{2}+\delta}}\right)^2}{2}\right\}\frac{e^{-\frac{\left(y_N^j\right)^2}{2}}}{\sqrt{2\pi}y_N^j}=\Phi_Be^{-x}.
			\end{equation}
			A simple computation shows that the limit in \eqref{eq13} is equal to
			\begin{equation}\label{eq14}
				\lim_{N\to\infty}\frac{2^{(1-a)N}}{\sqrt{2\pi} N^{\frac{1}{2}+\delta}}\sum_{j:\frac{j}{N^{\frac{1}{2}+\delta}}\in B\cap[-N^{\epsilon},N^{\epsilon}]}\frac{1}{\sqrt{2\pi}y_N^j}\exp\left\{-\frac{\sqrt{a}\beta_*j}{N^\delta}-\frac{j^2}{2N^{1+2\delta}}-\frac{\left(y_N^j\right)^2}{2}\right\},
			\end{equation}
			whenever the latter limit exists. Since $\epsilon<\frac{1}{2}$, we have that
			$$\lim_{N\to\infty}\max_{|j|<N^{\frac{1}{2}+\delta+\epsilon}}\frac{y_N^j}{\beta_*\sqrt{(1-a)N}}=1,$$
			and
			\begin{equation} 
				\lim_{N\to\infty}\max_{|j|<N^{\frac{1}{2}+\delta+\epsilon}} \frac{2^{(1-a)N}e^{-\frac{\left(y_N^j\right)^2}{2}}}{\beta_*\sqrt{2\pi}\sqrt{N}\exp\left\{\frac{\sqrt{a}\beta_* j}{N^\delta}-\frac{aj^2}{2(1-a)N^{1+2\delta}}\right\}}=e^{-x}.
			\end{equation}
			Then, in order to conclude the proof of \eqref{eq9}, we only need to show that
			\begin{equation}\label{eq2}
				\lim_{N\to\infty}\sum_{j:\frac{j}{N^{\frac{1}{2}+\delta}}\in B\cap[-N^{\epsilon},N^{\epsilon}]}\frac{1}{\sqrt{2\pi(1-a)}N^{\frac{1}{2}+\delta}}\exp\left\{-\frac{j^2}{2(1-a)N^{1+2\delta}}\right\}=\Phi_B,
			\end{equation}
			which holds because the left hand side in the limit above is a Riemann sum for the integral that defines $\Phi_B$.

		\end{proof}

		\begin{corollary}\label{corollary_lemma2} 
			Let $B$ be a Borel set of $\mathbb{R}$ and $\epsilon\in(0,\frac{1}{2})$, then we have
			\begin{align*}
				\lim_{N\rightarrow\infty}\sum_{j:\frac{j}{N^{\frac{1}{2}+\delta}}\in B\cap[-N^{\epsilon},N^{\epsilon}]}\mathbb{P}\left[u_N^{-1}(Z_N^j)>x\right] =\Phi_Be^{-x}.
			\end{align*}
		\end{corollary}

		\begin{proof}
			Applying logarithms in Lemma \ref{lemma2}, we obtain
			
			\begin{align*}
				\lim_{N\to\infty} \sum_{j:\frac{j}{N^{\frac{1}{2}+\delta}}\in B\cap [-N^\epsilon,N^\epsilon]}\log\left(1-\mathbb{P}\left[Z_N^j> u_N(x)\right]\right)=-\Phi_B e^{-x}.
			\end{align*}
			Similarly as in~(\ref{taylor}), we have that 
			$$\log\left(1-\mathbb{P}\left[Z_N^j> u_N(x)\right]\right)=-\mathbb{P}\left[Z_N^j> u_N(x)\right]-\frac{\mathbb{P}\left[Z_N^j> u_N(x)\right]^2}{2(1-\theta_N^j)^2},$$
			where $0<\theta_N^j<\mathbb{P}\left[Z_N^j> u_N(x)\right]$.
			Hence, it suffices to prove 
			$$\lim_{N\to\infty}\max_{j:|j|\leq N^{\frac{1}{2}+\delta+\epsilon}}\mathbb{P}\left[Z_N^j> u_N(x)\right]=0,$$
			or, equivalently 
			$$\lim_{N\to\infty}\mathbb{P}\left[Z_N^j\leq u_N(x)\right]=1 \text{ uniformly in } \left\{j:|j|\leq N^{\frac{1}{2}+\delta+\epsilon}\right\}.$$
			Recall, from the proof of Lemma \ref{lemma2}, that
			$$\mathbb{P}\left[Z_N^j\leq u_N(x)\right]=\mathbb{P}\left[X\leq y_N^j\right]^{2^{N_2}\mathbb{E}[\mathcal{C}(I_N^j)]}.$$
			Hence, it is enough to prove that 
			$$\lim_{N\to\infty}2^{N_2}\mathbb{E}[\mathcal{C}(I_N^j)]\log\left(1-\mathbb{P}\left[X>y_N^j\right]\right)=0 \text{ uniformly in } \left\{j:|j|\leq N^{\frac{1}{2}+\delta+\epsilon}\right\}.$$
			Note that the expression above is the same one appearing in the sum in \eqref{eq1}. In \eqref{eq2} we proved that the above limit equals 
			$$\lim_{N\to\infty}\frac{e^{-x}}{\sqrt{2\pi(1-a)}N^{\frac{1}{2}+\delta}}\exp\left\{-\frac{j^2}{2(1-a)N^{1+2\delta}}\right\},$$
			which can be readily checked to be zero uniformly in $\left\{j:|j|\leq N^{\frac{1}{2}+\delta+\epsilon}\right\}$.
		\end{proof}

		Recall from \eqref{def5} the definition of $Z_N^j$. For $B$ as in Lemma~\ref{lemma2}, let us relabel the indices $j$ belonging to the set  $\left\{j:\frac{j}{N^{\frac{1}{2}+\delta}}\in B\cap[-N^\epsilon,N^\epsilon]\right\}$ as $j_B(1),j_B(2),\cdots$ in such a way that $Z_N^{j_B(1)}>Z_N^{j_B(2)}>\cdots.$ We will omit $B$ in the notation $Z_N^{j_B(i)}$ when $B=\mathbb{R}$. The next result describes the scaling limit distribution of the vector $\left(Z_N^{j_B(1)},Z_N^{j_B(2)}\right)$.
		
		\begin{lemma}\label{lema_conjunta}
			Let $B$ be as in Lemma~\ref{lemma2} and $\epsilon\in(0,\frac{1}{2})$. Then, for $x\geq y$, we have
			$$\lim_{N\to\infty}\mathbb{P}\left[u_N^{-1}(Z_N^{j_B(1)})\leq x, u_N^{-1}(Z_N^{j_B(2)})\leq y\right]=\left[1+\Phi_B(e^{-y}-e^{-x})\right]e^{-\Phi_Be^{-y}}.$$
		\end{lemma}
		\begin{proof}
			\begin{align}\label{eq3}
				\mathbb{P}&\left[u_N^{-1}(Z_N^{j_B(1)})\leq x, u_N^{-1}(Z_N^{j_B(2)})\leq y\right]\nonumber\\
				&=\mathbb{P}\left[Z_N^{j_B(1)}\leq u_N(x), Z_N^{j_B(2)}\leq u_N(y)\right]\nonumber\\
				&=\mathbb{P}\left[Z_N^{j_B(1)}\leq u_N(y)\right]+\mathbb{P}\left[u_N(y)<Z_N^{j_B(1)}\leq u_N(x),Z_N^{j_B(2)}\leq u_N(y)\right]\nonumber\\
				&=\mathbb{P}\left[Z_N^{j_B(1)}\leq u_N(y)\right]+\sum_{j:\frac{j}{N^{\frac{1}{2}+\delta}}\in B\cap[-N^\epsilon,N^\epsilon]}\mathbb{P}\left[u_N(y)<Z_N^j\leq u_N(x)\right]\mathbb{P}\left[\max_{l\neq j}Z_N^l\leq u_N(y)\right].	
			\end{align}
			It readily follows from the arguments used in the proof of Lemma \ref{lemma2} that we have that 
			$$\displaystyle \lim_{N\to\infty}\frac{\mathbb{P}\left[\max_{l\neq j}Z_N^l\leq u_N(y)\right]}{\mathbb{P}\left[\max_{l}Z_N^l\leq u_N(y)\right]}=1 \text{ uniformly in } \left\{j:\frac{j}{N^{\frac{1}{2}+\delta}}\in B\cap[-N^\epsilon,N^\epsilon]\right\}.$$ 
			Then the limit of the expression in \eqref{eq3} is equal to
			\begin{align*}
				\lim_{N\to\infty}\mathbb{P}\left[ Z_N^{j_B(1)}\leq u_N(y)\right]\left(1+\sum_{j:\frac{j}{N^{\frac{1}{2}+\delta}}\in B\cap[-N^\epsilon,N^\epsilon]}\left(\mathbb{P}\left[Z_N^j\leq u_N(x)\right]-\mathbb{P}\left[Z_N^j\leq u_N(y)\right]\right)\right).
			\end{align*}
			The proof is concluded using Lemma \ref{lemma2} and Corollary \ref{corollary_lemma2}. 
		\end{proof}
		
		We next establish a convergence result for the variables $Z_N^j$, which constitutes a version of Theorem \ref{Thm_1} when we replace $Y_N^j$ by $Z_N^j$. 
		\begin{proposition}\label{casi_teorema}
			Let $\mathcal{P}=\{\xi_i:i\geq 1\}$  be the Poisson point process in Definition \ref{PPP}, with $K=1$. For $k\geq1$, let $W_1,\cdots,W_k$, be independent standard Gaussian variables with variance $1-a$, which are also independent of $\mathcal{P}$. Then 
			$$\left( \frac{j(1)}{N^{\frac{1}{2}+\delta}},\dots,\frac{j(k)}{N^{\frac{1}{2}+\delta}};\,
			u_N^{-1}\left(Z_N^{j(1)}\right),\dots,u_N^{-1}\left(Z_N^{j(k)}\right)\right)$$
			converges in distribution to $\left(W_1,\ldots,W_k;\,\xi_1,\ldots,\xi_k\right)$ as $N$ goes to infinity.
		\end{proposition}
		
		\begin{proof}
			It is enough to show that, for $\infty>d_1>c_1\geq d_2>c_2\geq\dots\geq d_k>c_k>-\infty$,  and $B_1, \dots,B_k$ disjoint intervals of $\mathbb{R}$, we have		
			\begin{align}\nonumber
				&\lim_{N\to\infty} \mathbb{P}\left[ 
				\frac{j(1)}{N^{\frac{1}{2}+\delta}}\in B_1,\dots,
				\frac{j(k)}{N^{\frac{1}{2}+\delta}}\in B_k;\, 
				u_N^{-1}\left(Z_N^{j(1)}\right)\in (c_1,d_1],\dots
				u_N^{-1}\left(Z_N^{j(k)}\right)\in (c_k,d_k] 
				\right]\\
				&= \left(\prod_{i=1}^k\Phi_{B_i} \right)\mathbb{P}\left[\xi_1\in(c_1,d_1],\dots,\xi_k\in(c_k,d_k]\right]\label{finito_dim} 
			\end{align}
			(see Theorem $2.2$ from \cite{Bill2013}, which may be applied to 
			$\big\{\mathbb{R}^{k}\setminus\{ (x_1,\dots,x_k)\in\mathbb{R}^{k}: x_i=x_j \text{ for some } i\neq j\}\big\}\times\mathbb{S}_k$, which has full measure under the limit law, where $\mathbb{S}_k=\{x\in\mathbb{R}^{k}:\,x_1>x_2>\cdots>x_k\}$).

			For short, let  $\widetilde{B}_k=(B_1\cup\cdots\cup B_k)^c$.  Recalling the definition of $j_B(i)$ above Lemma \ref{lema_conjunta}, the probability on the left hand side of (\ref{finito_dim}) equals
			\begin{align}\label{eq4}
				\mathbb{P}\left[\bigcap_{i=1}^k \left\{u_N^{-1}\left(Z_N^{j_{B_i}(1)}\right)\in[c_i,d_i]\right\},\bigcap_{i=1}^{k-1}\left\{u_N^{-1}\left(Z_N^{j_{B_k}(1)}\right)>u_N^{-1}\left(Z_N^{j_{B_i}(2)}\right)\right\}\right.,\nonumber\\
				\left.\left\{u_N^{-1}\left(Z_N^{j_{B_k}(1)}\right)>u_N^{-1}\left(Z_N^{j_{\widetilde{B}_k}(1)}\right)\right\}\right].
			\end{align}
			Let $f_{N,i}$ be the density of the random variable $u_N^{-1}\left(Z_N^{j_{B_i}(1)}\right)$ and $h_{N,i}$ be the joint density of the random vector $\left(u_N^{-1}\left(Z_N^{j_{B_i}(1)}\right),u_N^{-1}\left(Z_N^{j_{B_i}(2)}\right)\right)$. The probability in \eqref{eq4} then equals
			\begin{align}\label{eq5}
				\int_{c_k}^{d_k}\mathbb{P}\left[u_N^{-1}\left(Z_N^{j_{\widetilde{B}_k}(1)}\right)<x\right]\left(\prod_{i=1}^{k-1}\int_{-\infty}^{x}\int_{c_i}^{d_i}h_{N,i}(s,t)dsdt\right)f_{N,k}(x)dx.
			\end{align}
			By Lemmas \ref{lemma2} and \ref{lema_conjunta}, we have, respectively, that 
			$$\lim_{N\to\infty}	\mathbb{P}\left[u_N^{-1}\left(Z_N^{j_{\widetilde{B}_k}(1)}\right)<x\right] = e^{-(1-\sum_{i=1}^{k}\Phi_{B_i})e^{-x}}
			$$
			and 
			
			\begin{align*}
				\lim_{N\to\infty} \prod_{i=1}^{k-1}\int_{-\infty}^{x}\int_{c_i}^{d_i}h_{N,i}(s,t)dsdt=\prod_{i=1}^{k-1}\left(\Phi_{B_i}e^{-\Phi_{B_i}e^{-x}}\int_{c_i}^{d_i}e^{-s}ds\right).
			\end{align*}
			It promptly follows that the limit in \eqref{eq5} equals
			$$\left(\prod_{i=1}^k\Phi_{B_i}\right)\left(\prod_{i=1}^{k-1}\int_{c_i}^{d_i}e^{-s}ds\right)\left(\int_{c_k}^{d_k}e^{-x}e^{-e^{-x}}dx\right),$$
			which is readily checked to equal the right hand side of (\ref{finito_dim}).

		\end{proof}
		
		The next lemma states a deviation result for the number of indices $\sigma_1$ such that $\Xi^{(1)}_{\sigma_1}\in I_N^j$. In Proposition \ref{casi_teorema} we stated the convergence of the variables $Z_N^j$ and, with the next result, we will be able to prove the convergence of the variables $Y_N^j$.

		\begin{lemma}\label{second moment}
			For all $0<\epsilon<\frac{1}{2}$ and  $A>0$ we have 
			$$
			\lim_{N\to\infty} \sum _{|j|\leq N^{\frac{1}{2}+\epsilon+\delta}}\mathbb{P}\left[ \left|\mathcal{C}(I_N^j)- \mathbb{E}\left[\mathcal{C}(I_N^j)\right]\right|\geq A \mathbb{E}\left[\mathcal{C}(I_N^j)\right]\right] =0.
			$$
		\end{lemma}
		
		\begin{proof}
			Recall~(\ref{def3}). Note that 
			$$\mathcal{C}(I_N^j)=\displaystyle\sum_{\sigma_1\in\mathcal{V}_{N_1}}\mathbbm{1}_{\left\{\Xi^{(1)}_{\sigma_1}\in I_N^j\right\}}.$$
			Then, for any integer $j\in \left[-N^{\frac{1}{2}+\epsilon+\delta},N^{\frac{1}{2}+\epsilon+\delta}\right]$, and all large $N$, we get
			\begin{align}
				\mathbb{E}[\mathcal{C}(I_N^j)]&=2^{N_1}\mathbb{P}[\Xi^{(1)}_{\sigma_1}\in I_N^j]
				=2^{N_1}\int_{I_N^j}\frac{1}{\sqrt{2\pi}}e^{-x^2/2}dx\\
				&\geq \frac{2^{N_1}}{\sqrt{2\pi}N^{\frac{1}{2}+\delta}} 
				\exp\left\{-\frac{1}{2}\left(\sqrt{aN}\beta_* + \frac{j+1}{N^{\frac{1}{2}+\delta}}\right)^2\right\}\\
				&\geq\frac{2^{(p-a)N}}{2\sqrt{2\pi}N^{\frac{1}{2}+\delta}}
				\exp\Big\{-\frac{1}{2}\Big[\Big(\sqrt{aN}\beta_*+\frac{j+1}{N^{\frac{1}{2}+\delta}}\Big)^2
				-\big(\sqrt{aN}\beta_*\big)^2\Big]\Big\}\\
				&=2^{(p-a)N}
				\exp\big\{o(N)\big\},
			\end{align} 
			since $j\leq N^{\frac{1}{2}+\epsilon+\delta}$ and $\epsilon<\frac12$. We also have that
			$$\left(\mathcal{C}(I_N^j)\right)^2 = 
			\sum_{\sigma_1\in\mathcal{V}_{N_1}}\mathbbm{1}_{\left\{\Xi^{(1)}_{\sigma_1}\in I_N^j\right\}} + 
			\sum_{\sigma_1,\tilde\sigma_1 \in\mathcal{V}_{N_1}\atop\sigma_1\ne\tilde\sigma_1 }
			\mathbbm{1}_{\left\{\Xi^{(1)}_{\sigma_1}\in I_N^j\right\}}
			\mathbbm{1}_{\left\{\Xi^{(1)}_{\widetilde{\sigma}_1}\in I_N^j\right\}},
			$$
			and, then 
			\begin{align*}
				\mathbb{E}\left[\left(\mathcal{C}(I_N^j)\right)^2\right] &=\mathbb{E}\left[\mathcal{C}(I_N^j)\right] +  2^{N_1}\left(2^{N_1}-1\right)\mathbb{P}\left[\Xi^{(1)}_{\sigma_1}\in I_N^j\right]^2\\
				&= \mathbb{E}\left[\mathcal{C}(I_N^j)\right] + \left(\mathbb{E}\left[\mathcal{C}(I_N^j)\right]\right)^2 - \mathbb{E}\left[\mathcal{C}(I_N^j)\right]\mathbb{P}\left[\Xi^{(1)}_{\sigma_1}\in I_N^j\right].				 
			\end{align*}
			So we have
			\begin{align*} 
				\frac{\mathbb{E}\left[\left(\mathcal{C}(I_N^j)\right)^2\right]-\left(\mathbb{E}\left[\mathcal{C}\left(I_N^j\right)\right]\right)^2}{\left(\mathbb{E}\left[\mathcal{C}(I_N^j)\right]\right)^2}&= \frac{1-\mathbb{P}\left[X_{\sigma_1}\in I_N^j\right]}{\mathbb{E}\left[\mathcal{C}(I_N^j)\right]}\leq \frac{1}{\mathbb{E}\left[\mathcal{C}(I_N^j)\right]},
			\end{align*}
			and hence
			\begin{align*}
				&\sum _{|j|\leq N^{\frac{1}{2}+\epsilon+\delta}}\mathbb{P}\left[ \left|\mathcal{C}(I_N^j)- \mathbb{E}\left[\mathcal{C}(I_N^j)\right]\right|> A \mathbb{E}\left[\mathcal{C}(I_N^j)\right]\right]\\
				&\leq 	\sum _{|j|\leq N^{\frac{1}{2}+\epsilon+\delta}} 	\frac{\mathbb{E}\left[\left(\mathcal{C}(I_N^j)\right)^2\right]-\left(\mathbb{E}\left[\mathcal{C}(I_N^j)\right]\right)^2}{A^2\left(\mathbb{E}\left[\mathcal{C}(I_N^j)\right]\right)^2}\\
				&\leq \sum _{|j|\leq N^{\frac{1}{2}+\epsilon+\delta}} 	\frac{1}{A^2\mathbb{E}\left[\mathcal{C}(I_N^j)\right]}
				\leq \frac{2N^{\frac12+\epsilon+\delta}}{A^2 2^{(p-a)N}}
				\exp\big\{o(N)\big\},
			\end{align*}	
			which converges to zero as $N$ goes to infinity, since $p>a$.
		\end{proof}

		Recall from \eqref{defmax} and \eqref{def6} the definition of $M_N^j$ and $Y_N^j$, respectively. Next we prove that these two variables are {\em close} to each other. 
		\begin{lemma}\label{MN_approx_YN}
			For all $N>0$ and all $j\in\mathbb{Z}$, we have
			$$0\leq M_N^j-Y_N^j\leq\frac{\sqrt{a}}{N^{\frac{1}{2}+\delta}}.$$
		\end{lemma}
		
		\begin{proof}
			For $\sigma_1\in\mathcal{V}_{N_1}$ such that $\Xi^{(1)}_{\sigma_1}\in I_N^j$, we have
			\begin{align*}
				\sqrt{a}\Xi_{\sigma_1}^{(1)}+\sqrt{1-a}\Xi_{\sigma_1\sigma_2}^{(2)}\leq\sqrt{a}\left(\sqrt{aN}\beta_*+\frac{j+1}{N^{\frac{1}{2}+\delta}}\right)+\sqrt{1-a}\Xi_{\sigma_1\sigma_2}^{(2)}\leq Y_N^j+\frac{\sqrt{a}}{N^{\frac{1}{2}+\delta}}, 
			\end{align*}
			hence, taking maximum in $\sigma_1\in\mathcal{V}_{N_1}$ such that $\Xi_{\sigma_1}^{(1)}\in I_N^j$, we have 
			\begin{equation}\label{eq6}
				M_N^j\leq Y_N^j+\frac{\sqrt{a}}{N^{\frac{1}{2}+\delta}}.
			\end{equation}
			For the other inequality, let us call $\widetilde{\sigma}_1$ and $\widetilde{\sigma}_2$ the indices such that
			$$\max_{\sigma_1\in\mathcal{V}_{N_1}:\Xi^{(1)}_{\sigma_1}\in I_N^j}\Xi^{(2)}_{\sigma_1\sigma_2}=\Xi^{(2)}_{\widetilde{\sigma}_1\widetilde{\sigma}_2}.$$
			
			Then
			\begin{align*}
				M_N^j\geq \sqrt{a}\Xi^{(1)}_{\widetilde{\sigma}_1}+\sqrt{1-a}\Xi^{(2)}_{\widetilde{\sigma}_1\widetilde{\sigma}_2}
				\geq \sqrt{a}\left(\sqrt{aN}\beta_*+\frac{j}{N^{\frac{1}{2}+\delta}}\right)+\sqrt{1-a}\Xi^{(2)}_{\widetilde{\sigma}_1\widetilde{\sigma}_2},
			\end{align*}
			that is,
			\begin{equation}\label{eq7}
				M_N^j\geq Y_N^j,
			\end{equation}
			which completes the proof. 
		\end{proof}
		
		For our next result we will use a similar notation as the one used in the paragraph above Lemma \ref{lema_conjunta}. Let us consider $j(1), j(2),\cdots$ in the set $\left\{j:|j|\leq N^{\frac{1}{2}+\delta+\epsilon}\right\}$ such that $M_N^{j(1)}>M_N^{j(2)}>\cdots$. And let us call $\hat{\sigma}(i)$ the configuration in $\mathcal{V}_N$ such that $M_N^{j(i)}=\Xi_{\hat{\sigma}(i)}$, for each positive integer $i$.
		\begin{proposition}\label{Thm_1_inside_the_box} 
			Let $\mathcal{P}=\{\xi_i:i\geq 1\}$  be the Poisson point process in Definition \ref{PPP} and consider  $p>a$. Let $W_1,\cdots,W_k$, be independent centered Gaussian variables with variance $1-a$, which also are independent of the process $\mathcal{P}$. Then 
			$$
			\left( u_N^{-1}\left(M_N^{j(1)}\right), \Xi^{(1)}_{\hat{\sigma}_1(1)}-\sqrt{aN}\beta_*,\dots,u_N^{-1}\left(M_N^{j(k)}\right), \Xi_{\hat{\sigma}_1(k)}^{(1)}-\sqrt{aN}\beta_*\right)
			$$
			converges in distribution to $\left(\xi_1,W_1,\cdots,\xi_k,W_k\right)$ as $N$ goes to infinity.
		\end{proposition}
		
		\begin{proof}
			Since, for each $i$, $\Xi^{(1)}_{\tilde{\sigma}_1(i)}\in I_N^{j(i)}$, we have that
			$$ 0\leq \Xi^{(1)}_{\hat{\sigma}_1(i)}-\left[\sqrt{aN}\beta_*+\frac{j(i)}{N^{\frac{1}{2}+\delta}}\right]\leq \frac{1}{N^{\frac{1}{2}+\delta}}.$$
			Then, using Lemma \ref{MN_approx_YN}, it is enough to prove that
			$$\left(u_N^{-1}\left(Y_N^{j(1)}\right),\frac{j(1)}{N^{\frac{1}{2}+\delta}};\cdots;u_N^{-1}\left(Y_N^{j(k)}\right),\frac{j(k)}{N^{\frac{1}{2}+\delta}}\right)\to (\xi_1,W_1;\cdots;\xi_k,W_k)$$ 
			in distribution.
			
			As part of the strategy, we compare the distribution of $(Y_N^j)$ to that of a perturbation of $(Z_N^j)$.
			For $\eta>0$, let us define $Z_N^{j,\eta}$ as
			\begin{equation}
				Z_N^{j,\eta}=\sqrt{a}\left(\frac{j}{N ^{\frac{1}{2}+\delta}}+\sqrt{aN}\beta_c\right)+\sqrt{1-a}\max_{k=1,\dots,2^{(1-p)N}\eta\mathbb{E}[\mathcal{C}(I_N^j)]}\widetilde{\Xi}_{j,k}.
			\end{equation}
			Obviously adapting the proof of Proposition \ref{casi_teorema}, we find, under the same conditions of the latter result, that
			\begin{align}\label{eta}
				\lim_{N\to\infty}& \mathbb{P}\left[\bigcap_{i=1}^k\left\{ u_N^{-1}\left(Z_N^{j(i),\eta}\right)\in [c_i,d_i], \frac{j(i)}{N^{\frac{1}{2}+\delta}}\in B_i\right\} \right]\nonumber\\
				&= \eta^k
				\left(\prod_{i=1}^k\Phi_{B_i}\right)
				\left(\prod_{i=1}^{k-1}\int_{c_i}^{d_i}e^{-s}ds\right)\left(\int_{c_k}^{d_k}e^{-x}e^{-\eta e^{-x}}dx\right),
			\end{align}
			which may be readily checked to mean that 
			$$\left(u_N^{-1}\left(Z_N^{j(1),\eta}\right),\frac{j(1)}{N^{\frac{1}{2}+\delta}};\cdots;u_N^{-1}\left(Z_N^{j(k),\eta}\right),\frac{j(k)}{N^{\frac{1}{2}+\delta}}\right)\to (\xi^\eta_1,W_1;\cdots;\xi^\eta_k,W_k),$$ 
			where $\xi^\eta\equiv\xi+\log\eta$.

			Then, for $0<A<1$, using Lemma~\ref{second moment} and~(\ref{eta}), we get 
			\begin{align*}
				&\limsup_{N\to\infty}\mathbb{P}\left[\bigcap_{i=1}^k \left\{u_N^{-1}(Y_N^{j(i)})\geq x_i,\frac{j(i)}{N^{\frac{1}{2}+\delta}}\in B_i\right\}\right]\\
				=&\limsup_{N\to\infty}\mathbb{P}\left[\bigcap_{i=1}^k \left\{u_N^{-1}(Y_N^{j(i)})\geq x_i,\frac{j(i)}{N^{\frac{1}{2}+\delta}}\in B_i\right\}, \bigcap_{i=1}^k\left\{\left|\mathcal{C}(I_N^j)-\mathbb{E}\left[\mathcal{C}(I_N^j)\right]\right|<A\mathbb{E}\left[\mathcal{C}(I_N^j)\right]\right\}\right]\\
				\leq&\lim_{N\to\infty} \mathbb{P}\left[\bigcap_{i=1}^k \left\{u_N^{-1}(Z_N^{j(i),1+A})\geq x_i,\frac{j(i)}{N^{\frac{1}{2}+\delta}}\in B_i\right\} \right]
				=\left(\prod_{i=1}^{k}\Phi_{B_i}\right)\mathbb{P}\left[\bigcap_{i=1}^k\left\{\xi_i\geq x_i-\log(1+A)\right\}\right]
			\end{align*}
			Similarly, we get a lower bound for the liminf of the probability on the left hand side of the above expression by exchanging $A$ with $-A$. Since $A$ is arbitrary, the result promptly follows.
		\end{proof}	
		
		\begin{proposition}\label{point_process_convergence}
			For $0<\epsilon<\frac{1}{2}$ and $\delta>0$, the point process $\widehat{\mathcal{P}}_N$, defined as
			\begin{equation*}
				\widehat{\mathcal{P}}_N:=\sum_{j:|j|\leq N^{\frac{1}{2}+\epsilon+\delta}}\mathbbm{1}_{u_N^{-1}(M_N^j)},
			\end{equation*}
			converges in distribution to a Poisson point process $\mathcal{P}$ with intensity measure $e^{-x}dx$.
		\end{proposition}
		
		\begin{proof}
			Taking $B=\mathbb{R}$ in Corollary \ref{corollary_lemma2}, and using Lemmas \ref{second moment} and \ref{MN_approx_YN}, we have
			\begin{equation}\label{resultado1}
				\lim_{N\to\infty}\sum_{j:|j|\leq N^{\frac{1}{2}+\delta+\epsilon}}\mathbb{P}\left[M^j_N>u_N(x)\right] = e^{-x}.
			\end{equation}
			By Theorem A.1 of \cite{leadbetter2012extremes}, it is enough to show 
			\begin{enumerate}
				\item [(i)]   $\lim_{N\to\infty}\mathbb{E}\left[\widehat{\mathcal{P}}_N((c,d])\right]= \mathbb{E}[\mathcal{P}((c,d])]$ for any real numbers $c<d$.
				\item [(ii)] Let be $B=\cup_{i=1}^k(c_i,d_i]$  for $-\infty< c_1<d_1<c_2<d_2<\dots<c_k<d_k$, then
				$$
				\lim_{N\to\infty}\mathbb{P}\big[\widehat{\mathcal{P}}_N(B)=0\big]= \mathbb{P}\big[\mathcal{P}(B)=0\big].
				$$
			\end{enumerate}
			Proof of (i). By definition, we have
			\begin{align*}
				\mathbb{E}\left[\widehat{\mathcal{P}}_N((c,d])\right]&=\sum_{|j|\leq N^{\frac{1}{2}+\delta +\epsilon}}\mathbb{P}\big[M^j_N\in (u_N(c),u_N(d)]\big]\\
				&=\sum_{|j|\leq N^{\frac{1}{2}+\delta+\epsilon}}\mathbb{P}\big[M_N^j> u_N(c)]\big] -\sum_{|j|\leq N^{\frac{1}{2}+\delta +\epsilon}}\mathbb{P}\big[M_N^j > u_N(d)]\big],
			\end{align*} 
			which, by \eqref{resultado1}, converges to $e^{-c}-e^{-d}$. As we can easily see, $ \mathbb{E}[\mathcal{P}((c,d])]=e^{-c}-e^{-d}$, and (i) follows.
			
			Proof of (ii). We get, by the independence of the variables $M_N^j$ as $j$ varies, that
			\begin{align*}
				\mathbb{P}\left[\widehat{\mathcal{P}}_N(B)=0\right]&=\mathbb{P}\left[ M_N^j\notin u_N(B), \forall j: |j|\le N^{\frac{1}{2}+\epsilon+\delta} \right]\\
				&=\exp\left\{\sum_{|j|\leq N^{\frac{1}{2}+\delta+\epsilon}}\log\left[1-\mathbb{P}[M_N^j\in u_N(B)]\right]\right\}.
			\end{align*}
			We can use a Taylor's Theorem argument similar to the one we used in the proof of Lemma \ref{lemma2}, and obtain
			\begin{align*} 
				\lim_{N\to\infty}\mathbb{P}\left[\widehat{\mathcal{P}}(B)=0\right]&=\lim_{N\to\infty} \exp\left\{ -\sum_{|j|\leq N^{\frac{1}{2}+\epsilon+\delta}} \mathbb{P}\left[ M_N^j \in u_N(B)\right] \right\}\\
				&=\lim_{N\to\infty}\exp\left\{ -\sum_{i=1}^{k}\sum_{|j|\leq N^{\frac{1}{2}+\epsilon+\delta}} \mathbb{P}\left[ M_N^j \in (u_N(c_i),u_N(d_i)]\right]\right\}\\
				&=\exp\left\{-\sum_{i=1}^k(e^{-c_i}-e^{-d_i})\right\}.
			\end{align*}
			On the other hand, we also have that $\mathbb{P}\left[\mathcal{P}(B)=0\right]=\exp\left\{-\sum_{i=1}^k(e^{-c_i}-e^{-d_i})\right\}$. And this completes the proof.
		\end{proof}

		We need a final piece to complete the proof of Theorem \ref{Thm_1}. Up to this point, we have been assuming that our indices $j$ are in the interval $[-N^{\frac{1}{2}+\delta+\epsilon},N^{\frac{1}{2}+\delta+\epsilon}]$. Now we would like to remove this restriction, and, for that, let us recall, from the paragraph above Proposition \ref{Thm_1_inside_the_box}, that $\hat{\sigma}(i)$ is defined as the configuration that satisfies $M_N^{j(i)}=\Xi_{\hat{\sigma}(i)}$, where the indices $j(i)\in\left\{j:|j|\leq N^{\frac{1}{2}+\delta+\epsilon}\right\}$ are such that $M_N^{j(1)}>M_N^{j(2)}>\cdots$. Also, without any restriction in $j$,  we say that $\sigma(i)$ is the configuration in which the $i^{th}$ maximum of $M_N^j$ is attained, that is, $M_N^{j(i)}=\Xi_{\sigma(i)}$. The next lemma shows that the mentioned assumption is removable.

		\begin{lemma} \label{equals_maximum} For any $k\ge 1$ we have
			\begin{equation*}
				\lim_{N\to\infty}\mathbb{P}\left[\Xi_{\sigma(1)}=\Xi_{\hat{\sigma}(1)},\dots,\Xi_{\sigma(k)}=\Xi_{\hat{\sigma}(k)}\right]=1.
			\end{equation*}
		\end{lemma}
		\begin{proof}
			The proof follows readily from the fact that, on the one hand, $\widehat{\mathcal{P}}_N$ is (clearly) dominated by ${\mathcal{P}}_N$, and on the other hand, from Proposition~\ref{point_process_convergence} and~(\ref{BK}), both point processes have the same limit.
		\end{proof}	
		
		\begin{proof}[Proof of Theorem \ref{Thm_1}]
			Note that the case $a<p$ follows readily from Proposition \ref{Thm_1_inside_the_box} and Lemma \ref{equals_maximum}. It remains to prove the case $a=p$, for which all the results for $a<p$ extend with the corresponding changes in statements, and similar proofs. For that reason we will only point out some little modifications in the results above. 
			
			In the intervals $I_N^j$, defined in \eqref{def3}, let us only consider $j\leq -1$. Then, following the lines of the proof of the case $a<p$, we obtain a similar result as the one in Proposition \ref{casi_teorema}, where, the independent and centered Gaussian variables $W_1,\cdots,W_k$ are now conditioned on being negative and $K=\frac{1}{2}$ in the Poisson point process $\mathcal{P}=\{\xi_i:i\geq 1\}$ of Definition \ref{PPP}. Also, considering $-\frac{1}{2}<\delta<0$ in the intervals $I_N^j$, the corresponding result obtained in Lemma \ref{second moment} also holds.
			
			As before, given \eqref{BK}, it will be enough to prove the following analogous result to Proposition \ref{point_process_convergence}: 
			$$\sum_{j: N^{\frac{1}{2}+\epsilon+\delta}\leq j\leq -1}\mathbbm{1}_{u_N^{-1}(M_N^j)}$$
			converges in distribution to a Poisson point process $\mathcal{P}$ with intensity measure $\frac{1}{2}e^{-x}dx$, which will be a straightforward adaptation of the previous proof.	
		\end{proof}

		\section{Proof of Theorem \ref{Thm_above_FT} }\label{sec:below FT}
		\setcounter{equation}{0}
		
		We recall that we are resorting to Skorohod's Representation Theorem to have our environment family of random variables $\Xi$ realized in a space where the convergence stated in Theorem \ref{Thm_1} (for the original space) holds almost surely; i.e.,
		we have that the convergences in~\eqref{BK},~\eqref{convergence_exp} and 
		\begin{equation}\label{limit_thm_1}
			\lim_{N\to\infty}\left[\Xi^{(1)}_{\sigma_1(k)}-\sqrt{aN}\beta_*\right]=W_k\, , \text{ for all } k\ge 1,\, 
		\end{equation}
		hold almost surely; we further assume that we are in the event of that space where those convergences take place everywhere.

		We start with some definitions. Consider the following sets:
		\begin{equation*}
			I:=\{ i\geq 1 : W_i>L\} \, \text{ and } J:=\{ j\geq 1 : W_j <L  \}\, ,
		\end{equation*}
		and, given $M\ge 1$, the subsets $I_M:=\{ i_1, \dots, i_M\}\subset I$ and  $J_M:=\{j_1,\dots, j_M\}\subset J$ characterized by:
		\begin{enumerate}
			\item [(a)] $ W_{i_1}< W_{i_2}<\dots < W_{i_M} $ and $ W_{j_1}< W_{j_2}<\dots W_{j_M} $, 
			\item [(b)] $\max I_M< \min I\setminus I_M$ and $\max J_M < \min J\setminus J_M$. 
		\end{enumerate}	
		Let us denote by $X^N_M$ the restriction of the process $X^N$ to $I_M\cup J_M$. 
		
		\begin{proposition}\label{prop_4.1} Given $M\ge 1$ such that $\ell\in I_M$,
			we have for all $t>0$ that
			\begin{equation*}
				\lim_{N\to\infty}\frac{1}{c_Nt}\int_{0}^{c_Nt}\mathbb{1}_{\{X_M^N(s)=\ell\}}ds = \frac{\gamma(\ell)}{\sum_{m=1}^{M}\gamma(i_m)}\, ,
			\end{equation*}
			in probability.
		\end{proposition}	
		
		\begin{proposition} \label{prop_4.2} Given $t> 0$, let $T_M^{N,out}(t)$ be the time spent by $X^N$ outside $I_M\cup J_M$ up to time $c_Nt$. Then for any $\lambda >0$ we have
			\begin{equation*}
				\lim_{M\to\infty}\limsup_{N\to\infty}\mathbb{P}\left[\frac{1}{c_Nt}T_M^{N,out}(t)>\lambda\right]= 0\, .
			\end{equation*}
		\end{proposition}	
		
		\begin{proof}[Proof of Theorem \ref{Thm_above_FT}] For $M>0$ fixed and $\ell\in I_M$, let us write
			\begin{align}\label{eq_4.1}
				\int_{0}^{c_Nt}\mathbb{1}_{\{X^N(s)=\ell\}}ds &= \int_{0}^{c_Nt-T_M^{N,out}(t)}\mathbb{1}_{\{X_M^N(s)=\ell\}}ds\nonumber\\
				&= \int_{0}^{c_Nt}\mathbb{1}_{\{X_M^N(s)=\ell\}}ds-\int_{c_Nt-T_M^{N,out}(t)}^{c_Nt}\mathbb{1}_{\{X_M^N(s)=\ell\}}ds\, .
			\end{align}
			In view of 
			\begin{equation*}
				\lim_{M\to\infty} \sum_{m=1}^{M}\gamma(i_m)=\sum_{i: W_i>L}\gamma(i)\, ,
			\end{equation*}
			given any $\lambda$ positive, we can find some $M_0$ such that for all $M>M_0$ we get
			\begin{align*}
				\mathbb{P}&\left[\left|\frac{1}{c_Nt}\int_{0}^{c_Nt}\mathbb{1}_{\{X^N(s)=\ell\}}ds - \frac{\gamma(\ell)}{\sum_{i: W_i>L}\gamma(i)}\right|>\lambda\right]\\
				&\le  \mathbb{P}\left[\left|\frac{1}{c_Nt}\int_{0}^{c_Nt}\mathbb{1}_{\{X^N(s)=\ell\}}ds - \frac{\gamma(\ell)}{\sum_{m=1}^{M}\gamma(i_m)}\right|>\frac{\lambda}{2}\right]\, .
			\end{align*}	
			By \eqref{eq_4.1}, the probability in the right hand side above has the upper bound
			\begin{equation*}
				\mathbb{P}\left[\left|\frac{1}{c_Nt}\int_{0}^{c_Nt}\mathbb{1}_{\{X_M^N(s)=\ell\}}ds - \frac{\gamma(\ell)}{\sum_{m=1 }^{M}\gamma(i_m)}\right|>\frac{\lambda}{4}\right]  + \mathbb{P}\left[\frac{1}{c_Nt}T_M^{N,out}(t)>\frac{\lambda}{4}\right]\, .
			\end{equation*}
			Taking $\limsup$ as $N\to\infty$ and then $M\to\infty$, using Proposition \ref{prop_4.1} and \ref{prop_4.2} we obtain Theorem \ref{Thm_above_FT}.
		\end{proof}	
		
		As will become clear in the proof, it will be sufficient to show Proposition \ref{prop_4.1} and \ref{prop_4.2} for $t=1$.
		
		\subsection{Proof of Proposition \ref{prop_4.1}}\label{sub_seq_4.1}
		
		Let us denote by $\left\{J^{1,N}(j), j\geq 0\right\}$ and $\left\{J^{2,N}(j), j\geq 0\right\}$ two independent, discrete time, simple, symmetric random walks evolving in $\mathcal{V}_{N_1}$ and $\mathcal{V}_{N_2}$, respectively. For any given probability measure $\mu_i$ defined in $\mathcal{V}_{N_i}$, we set
		$$\mathbb{P}_{\mu_i}[J^{i,N}(0)=\sigma_i]=\mu_i(\sigma_i).$$
		Let us denote by  $\pi_1$ and $\pi_2$ the uniform distribution in $\mathcal{V}_{N_1}$ and $\mathcal{V}_{N_2}$, respectively. We will assume that $J^{1,N}$ starts from $\pi_1$ and $J^{2,N}$ starts from $\pi_2$. Let us consider a third probability space $(\Omega,\mathcal{F},\widetilde{\mathbb{P}})$ in which a family of i.i.d. mean one exponential random variables $\left\{T_j,j\geq 0\right\}$ is defined. We will describe the evolution of $\sigma^N(t)$ through the product probability $\mathbb{P}=\mathbb{P}_{\mu_1}\times\mathbb{P}_{\mu_2}\times\widetilde{\mathbb{P}}$. 
		
		For the sake of simplicity we will denote 
		$$\mathbb{P}_{\mu_1\times\mu_2}=\mathbb{P}_{\mu_1}\times\mathbb{P}_{\mu_2}\times\widetilde{\mathbb{P}}.$$
		
		In some steps of the proof, we will be dealing with situations in which $J^{1,N}$ starts from certain configuration $\sigma_1(i_m)$ and $J^{2,N}$ starts from the uniform distribution $\pi_2$, in which case, we adopt the abbreviated notation 
		\begin{equation}\label{def_prob}
			\mathbb{P}_m=\mathbb{P}_{\delta_{\sigma_1(i_m)}\times\pi_2},
		\end{equation}
		where $\delta_{\sigma_1}$ denotes the Dirac measure in $\mathcal{V}_{N_1}$ concentrated in $\sigma_1$. Also, we will replace the notation $\mathbb{P}_{\delta_{\sigma_1}}$ by  $\mathbb{P}_{\sigma_1}$.

		Recalling the transition rate defined in \eqref{dynamics}, the first component of $\sigma^N$ changes after a geometric number of jumps in the second component, and on each one of those jumps, the amount of time that the process spends has exponential distribution. So, when $X^N$ arrives at some state $i$, the time spent until it decides to jump is distributed as a sum of these exponential random variables:  
		\begin{equation}\label{eq_4.1.1}
			H^N(i)=\sum_{j=0}^{G^N(i)-1}\frac{N}{N_2}\frac{e^{\beta\sqrt{(1-a)N}\Xi_{\sigma_1(i)J^{2,N}(j)}^{(2)}}}{1+\frac{N_1}{N_2}e^{-\beta\sqrt{aN}\Xi_{\sigma_1(i)}^{(1)}}}T_j=	\sum_{j=0}^{G^N(i)-1}\frac{N}{N_1}\frac{e^{\beta\sqrt{N}\Xi_{\sigma_1(i)J^{2,N}(j)}}}{\mu_N(i)}T_j\, ,
		\end{equation}
		where $G^N(i)=G^N(\sigma_1(i))$ is a Geometric random variable with mean $\mu_N(i)$, defined as 
		\begin{equation}\label{def_mu}
			\mu_N(i)=1+\frac{N_2}{N_1}e^{\beta\sqrt{aN}\Xi_{\sigma_1(i)}^{(1)}}.
		\end{equation} 
		
		Recall the conditions (a) and (b) above Proposition \ref{prop_4.1}. In view of the independence and continuity of the distribution of the variables $W_i,\,i\geq1,$
		$$
		\sigma_1(i_1),\dots,\sigma_1(i_M),\sigma_1(j_1),\dots,\sigma_1(i_M),
		$$
		are well defined and distinct almost surely. Keeping this in mind, let us denote by $H^N_{m,\ell}$ the time spent by $X^N$ at site $i_m\in I_M$ in its $\ell^{th}$ visit. Note that, for each $m=1,\dots,M$, the variables $\{H^N_{m,\ell}; \ell\ge 1\}$ have the same distribution as $H^N(i_m)$, defined in \eqref{eq_4.1.1}; however, they are not independent, indeed they depend on the position of $J^{2,N}$ in each arrival to $i_m$, and these positions are not independent.

		Given a simple, symmetric random walk $J^{1,N}$ on $\mathcal{V}_{N_1}$, 
		let us consider the sequence of  times when $J^{1,N}$ visits $\sigma_1(i_1)$
		\begin{equation}\label{tau}
			\tau^{k}:=\inf\{n>\tau^{k-1}: J^{1,N}(n)=\sigma_1(i_1)\},\, k\geq1,
		\end{equation}
		with $\tau^{0}=-1$, and for $1\le i\le 2^N$,
		{\begin{equation}\label{S}
				S^k_i=\sum_{n=0}^{\tau^k}\mathbb{1}_{\{J^{1,N}(n)=\sigma_1(i)\}},
			\end{equation}
			the number of visits of $J^{1,N}$ to $\sigma_1(i)$ up to $\tau^k$.}
		
		Let us now consider the sequence of times spent by{ 
			$X^N$ on $i_m$ between returns to $i_1$}, that is, for $k\geq1$, set
		\begin{equation}\label{F_m}
			F^N_{m,k}:=\sum_{\ell=S^{k-1}_{i_m}+1}^{S^k_{i_m}}H^N_{m,\ell}\, .
		\end{equation}
		Similarly, let $\widehat{H}^N_{m,\ell}$ be the time spent by $X^N$ at site $j_m\in J_M$ in its $\ell^{th}$ visit, and
		\begin{equation}\label{Q_m}
			Q^N_{m,k}:=\sum_{\ell=S_{j_m}^{k-1}+1}^{S^k_{j_m}}\widehat{H}^N_{m,\ell} \, ,
		\end{equation}
		is the time spent by{
			$X^N$ on $j_m$ between returns to $i_1$}.
		
		
		Finally, let us consider the sequence of times that $X^N$ spends in all sites $i_m$ and $j_m$, for  $m=1,\dots,M$, between returns to $i_1$:
		\begin{equation}\label{R^N}
			R^N_k:=\sum_{m=1}^M\left[F^N_{m,k}+Q^N_{m,k}\right],\, k\geq1.
		\end{equation}
		
		We note that for every $m$, the random variables $F^N_{m,k}$, $k\geq2$, are identically distributed among themselves, and identically distributed to $F^N_{m,1}$ under ${\mathbb P}_1$.
		The same holds for $Q^N_{m,k}$, $k\geq1$, and $R^N_{k}$, $k\geq1$.

		For the next lemma, let us define $\mathbb{E}_m$ the expectation with respect to the probability $\mathbb{P}_m$ defined in \eqref{def_prob}.
		\begin{lemma} \label{lemma_4.1.1} For $m=1,\dots,M$, we have
			\begin{equation*}
				\lim_{N\to\infty}\frac{\mathbb{E}_1\left[F^N_{m,1}\right]}{\mathbb{E}_1\left[R^N_1\right]}= \frac{\gamma(i_m)}{\sum_{m=1}^{M}\left[\gamma(i_m)+\gamma(j_m)\right]}\, .
			\end{equation*}
		\end{lemma}	
		
		\medskip
		
		\begin{proof}
			For $\ell\geq1$, let us define 
			\begin{equation}\label{gls}
				\mathcal{G}_\ell=\sum_{n=0}^{\tau_\ell(m)-1}G_n^N(J^{1,N}(n)),\,\,\, g_\ell=G_{\tau_\ell(m)}^N(i_m),
			\end{equation}
			where $\tau_\ell(m)$ is the number of jumps that $J^{1,N}$ executes up to its $\ell^{th}$ visit to $\sigma_1(i_m)$ and $\{ G^N_n(i): n\ge 0,\,i=1,\ldots,2^{N_1}\}$ is a sequence of independent, Geometric random variables with mean $\mu_N(i)$, respectively, defined in \eqref{def_mu}\footnote{If $\tau_1(m)=0$, then $\mathcal{G}_1=-1$.}. 
			By \eqref{eq_4.1.1}, we can write
			\begin{equation}\label{hml}
				H^N_{m,\ell}\overset{d}{=}\frac{N}{N_1\mu_N(i_m)}\sum_{j=1}^{g_\ell}
				e^{\beta\sqrt{N}\Xi_{\sigma_1(i_m)J^{2,N}(\mathcal{G}_\ell+j)}}T_{\mathcal{G}_\ell+j}\, ,
			\end{equation}
			
			Conditioning on  $J^{1,N}=\left\{J^{1,N}(n) : n\geq 1 \right\}$ and 
			$\mathcal{G}=\left\{\mathcal{G}_\ell : \ell\geq 1 \right\}$, by \eqref{F_m}, we have that 
			\begin{equation*}
				\mathbb{E}_1\left[F^N_{m,1}\right]=\frac{N}{N_1\mu_N(i_m)}
				\mathbb{E}_1\left[\sum_{\ell=1}^{S^1_{i_m}}\sum_{j=1}^{g_\ell}
				\mathbb{E}_1\left[e^{\beta\sqrt{N}\Xi_{\sigma_1(i_m)J^{2,N}(\mathcal{G}_\ell+j)}}
				T_{\mathcal{G}_\ell+j}\Big| J^{1,N},\mathcal{G}\right]\right]\, .
			\end{equation*}
			Recall that the random walks $J^{1,N}$ and $J^{2,N}$ are independent from each other and are also independent from $\mathcal{G}$. Then  
			\begin{align*}
				\mathbb{E}_1\left[e^{\beta\sqrt{N}\Xi_{\sigma_1(i_m)J^{2,N}(\mathcal{G}_\ell+j)}}
				T_{\mathcal{G}_\ell+j}\Big|J^{1,N},\mathcal{G}\right]&=\sum_{\sigma_2\in\mathcal{V}_{N_2}}
				e^{\beta\sqrt{N}\Xi_{\sigma_1(i_m)\sigma_2}}
				\mathbb{P}_{\pi_2}\left[J^{2,N}(\mathcal{G}_\ell+j)=\sigma_2\right]\\
				&=\frac{1}{2^{N_2}}\sum_{\sigma_2\in\mathcal{V}_{N_2}}e^{\beta\sqrt{N}\Xi_{\sigma_1(i_m)\sigma_2}}\, .
			\end{align*}
			It is known  from elementary theory of Markov chains  (see e.g.~(7.17) in~\cite{fg2018})
			\begin{equation}\label{nor}
				\mathbb{E}_{1}\left[S^1_{i_m}\right]~=~1. 
			\end{equation}
			Hence, 
			\begin{equation}\label{expectationF_m}
				\mathbb{E}_1\left[F^N_{m,1}\right]=
				\frac{N}{N_1}\frac{\widetilde{\mathbb{E}}\left[g_1\right]}{\mu_N(i_m)}\frac{1}{2^{N_2}}
				\sum_{\sigma_2\in\mathcal{V}_{N_2}}e^{\beta\sqrt{N}\Xi_{\sigma_1(i_m)\sigma_2}}=
				\frac{N}{N_1}\frac{1}{2^{N_2}}\sum_{\sigma_2\in\mathcal{V}_{N_2}}e^{\beta\sqrt{N}\Xi_{\sigma_1(i_m)\sigma_2}} \, .
			\end{equation}
			Analogously, we get
			\begin{equation}\label{expectationQ_m}
				\mathbb{E}_1\left[Q^N_{m,1}\right] = \frac{N}{N_1}\frac{1}{2^{N_2}}\sum_{\sigma_2\in\mathcal{V}_{N_2}}e^{\beta\sqrt{N}\Xi_{\sigma_1(j_m)\sigma_2}} \, .
			\end{equation}
			Hence 
			\begin{equation}\label{expectationRN}
				\mathbb{E}_1\left[R^N_1\right]=\frac{N}{N_1}\frac{1}{2^{N_2}}
				\sum_{m=1}^{M}\left[\sum_{\sigma_2\in\mathcal{V}_{N_2}}e^{\beta\sqrt{N}\Xi_{\sigma_1(i_m)\sigma_2}}	+\sum_{\sigma_2\in\mathcal{V}_{N_2}}e^{\beta\sqrt{N}\Xi_{\sigma_1(j_m)\sigma_2}}\right]\, .
			\end{equation}
			Finally, using the normalization $u_N^{-1}$, and recalling the definition \eqref{gamma}, we have
			\begin{align}\label{tch}
				\lim_{N\to\infty}\frac{\mathbb{E}_1\left[F^N_{m,1}\right]}{\mathbb{E}_1\left[R^N_1\right]}&=\lim_{N\to\infty}\frac{\sum_{\sigma_2\in\mathcal{V}_{N_2}}\gamma^N(\sigma_1(i_m)\sigma_2)}{\sum_{m=1}^M\sum_{\sigma_2\in\mathcal{V}_{N_2}}\left[\gamma^N(\sigma_1(i_m)\sigma_2)+\gamma^N(\sigma_1(j_m)\sigma_2)\right]}\nonumber\\
				&= \frac{\gamma(i_m)}{\sum_{m=1}^{M}\left[\gamma(i_m)+\gamma(j_m)\right]}\, \, ,
			\end{align}
			where we have used \eqref{convergence_exp} (in the strong form mentioned at the beginning of the section) in the past passage.
		\end{proof}

		%
		%

		Let $R^N_0$ be the time spent by $X^N_M$ until its first time out of $i_1$. Notice that $R^N_k$ is the time that $X^N_M$ spends between the $(k-1)^{st}$ and $k^{th}$ visit to $i_1$, $k\geq 1$. 
		
		
		\begin{lemma}\label{lemma_4.1.2} Given $\delta>0$ let us define 
			\begin{equation}\label{b_N}
				b_N:=\left\lfloor  \frac{\delta c_N}{\mathbb{E}_1\left[R^N_1\right]} \right\rfloor\, .
			\end{equation}
			Then, for any fixed $m=1,\dots,M$, we have
			\begin{equation}\label{eq_4.1.3}
				\lim_{N\to\infty}\frac{1}{b_N}\sum_{k=1}^{b_N} \frac{F^N_{m,k}}{\mathbb{E}_1[R^N_1]} = \frac{\gamma(i_m)}{\sum_{m=1}^{M}\left[\gamma(i_m)+\gamma(j_m)\right]}\, 
			\end{equation}
			and
			\begin{equation}\label{eq_4.1.4}
				\lim_{N\to\infty}\frac{1}{b_N}\sum_{k=1}^{b_N} \frac{Q^N_{m,k}}{\mathbb{E}_1[R^N_1]} = 0\, .
			\end{equation}
			Both convergences above hold in probability.
		\end{lemma}	
		
		
		\begin{proof}		
			We start with \eqref{eq_4.1.3}. Notice that the variables $\{F^N_{m,k}:k\geq 1\}$ have the same distribution, but 
			they are not independent, the dependence coming from the correlations along the trajectory of the random walk $J^{2,N}$.
			In order to control these correlations, we proceed as follows. 
			
			%
			We start with the following decomposition:
			\begin{equation*}
				F^N_{m,k}:=F^{1,N}_{m,k}+F^{2,N}_{m,k}\, ,
			\end{equation*}
			where $F_{m,k}^{1,N}$ registers the first $N^3$ steps of the random walk $J^{2,N}$ only, that is,
			\begin{equation*}
				F^{1,N}_{m,k}\overset{}{=}\frac{N}{N_1\mu_N(i_m)}\sum_{\ell=S^{k-1}_{i_m}+1}^{S^k_{i_m}}
				\left(\sum_{j=1}^{(K^N_\ell-1)\wedge g_\ell}
				e^{\beta\sqrt{N}\Xi_{\sigma_1(i_m)J^{2,N}({\mathcal{G}}_{\ell}+j)}}T_{{\mathcal{G}}_{\ell} +j}\right),
			\end{equation*}	
			where $K^N_\ell=N^3-\Ups_\ell$, and $\Ups_0,\Ups_1,\ldots$ is a family of iid Bernoulli($\frac12$) random variables, independent of everything else.
			
			We  allow the $N^3$ steps to $J^{2,N}$ in order to enable a coupling, after those many steps, to the uniform invariant measure;
			the $\Ups$'s comprise another enabler of such a coupling, as it helps break the periodicity of $J^{2,N}$ --- see paragraph below~\eqref{BG}. The coupled process does not exhibit the above mentioned correlations.

			We next show that $F^{1,N}_{m,\cdot}$ makes a negligible contribution to the expression whose limit is taken in \eqref{eq_4.1.3}.
			Let us write
			\begin{equation}
				\mathbb{E}_1\left[\frac{1}{b_N}\sum_{k=1}^{b_N}  \frac{F^{1,N}_{m,k}}{\mathbb{E}_1[R_1^N]}\right] = \frac{\mathbb{E}_1\left[F^{1,N}_{m,1}\right]}{\mathbb{E}_1\left[R^N_1\right]}\, .
			\end{equation}
			Reproducing the estimate of the expectation of $F_m^N$ obtained in \eqref{expectationF_m}, we get
			\begin{equation*}
				\mathbb{E}_1\left[F^{1,N}_{m,1}\right]\le \frac{N}{N_12^{N_2}}\frac{N^3}{\mu_N(i_m)}\sum_{\sigma_2\in\mathcal{V}_{N_2}}e^{\sqrt{N}\Xi_{\sigma_1(i_m)\sigma_2}}\, .
			\end{equation*}
			Therefore, using \eqref{expectationRN} and the normalization $u_N^{-1}$, we obtain
			\begin{equation*}
				\frac{\mathbb{E}_1\left[F^{1,N}_{m,1}\right]}{\mathbb{E}_1\left[R^N_1\right]}
				\le\frac{\sum_{\sigma_2\in\mathcal{V}_{N_2}}\gamma^N(\sigma_1(i_m)\sigma_2)}{\sum_{m=1}^M\sum_{\sigma_2\in\mathcal{V}_{N_2}}\left[\gamma^N(\sigma_1(i_m)\sigma_2)+\gamma^N(\sigma_1(j_m)\sigma_2) \right]}\,\frac{N^3}{\mu_N(i_m)}\, .
			\end{equation*}
			Using \eqref{convergence_exp} and recalling, from \eqref{def_mu}, the definition of $\mu_N(i_m)$, we have that, by \eqref{limit_thm_1}, the last factor above goes to zero as $N\to\infty$, which proves that  
			\begin{equation}\label{tch1}
				\lim_{N\to\infty}\frac{1}{b_N}\sum_{k=1}^{b_N} \frac{F^{1,N}_{m,k}}{\mathbb{E}_1[R^N_1]}=0\, ,
			\end{equation}	
			in probability. 
			%
			%
			It remains to show that 
			\begin{equation}\label{remainsF2}
				\lim_{N\to\infty}\frac{1}{b_N}\sum_{k=1}^{b_N} \frac{F^{2,N}_{m,k}}{\mathbb{E}_1[R^N_1]}= \frac{\gamma(i_m)}{\sum_{m=1}^{M}\left[\gamma(i_m)+\gamma(j_m)\right]}\, ,
			\end{equation}	
			in probability.

			Let us recall Lemma 3.1 of \cite{bovier_veronique}, which says that for a simple symmetric random walk $\{J^N(k):k\geq 0\}$ on $\mathcal{V}_N$, where $\pi$ is the uniform distribution on $\mathcal{V}_N$, $\theta_N := \frac{3\ln2}{2}N^2$, $\sigma,\widetilde{\sigma}\in\mathcal{V}_N$, and $i\ge 1$, we have
			\begin{equation}\label{BG}
				\left| \sum_{\ell=0}^{1}\mathbb{P}_{\pi}[J^N(\theta_N + i+\ell)=\tilde{\sigma}, J^N(0)=\sigma] - 2\pi(\sigma)\pi(\tilde{\sigma}) \right|\leq 2^{-3N+1}\, .
			\end{equation}
			
			
			In the event that $g_\ell\geq N^3$, we couple $J^{2,N}(\GG_\ell+K^N_\ell)$ to $U^N_\ell$, where $U^N_0,U^N_1,\ldots$ are iid random variables which are uniformly distributed in $\mathcal{V}_{N_2}$.
			(\ref{BG}) and $\beta<\beta_{FT}$ imply that the coupling may be set so that it holds for all $\ell\leq S^{b_N}_{i_m}$, for $b_N$ defined in \eqref{b_N}, with probability vanishingly close to 1 as $N\to\infty$.

			We may thus replace $J^{2,N}(\GG_\ell+K^N_\ell)$ by $U^N_\ell$ 
			in $F^{2,N}_{m,k}$, thus obtaining an iid family $\breve F^{2,N}_{m,k}$, $1\leq k\leq b_N$, and it is enough to 
			establish~(\ref{remainsF2}) with $\breve F^{2,N}_{m,k}$ replacing $F^{2,N}_{m,k}$. The modification of $J^{2,N}$ and $J^{N}$ produced by these replacements will denoted by $\breve J^{2,N}$ and $\breve J^{N}$, respectively. We notice that $\breve F^{2,N}_{m,k},\,k\geq1$, are iid.
			
			Then \eqref{remainsF2} follows from Chebyshev's inequality, once we use~(\ref{tch}, \ref{tch1}), and show that
			\begin{align}
				&\lim_{N\to\infty}\frac{1}{b_N}\frac{\mathbb{E}_1\left[\left(\breve F^{2,N}_{m,1}\right)^2\right]}{\left(\mathbb{E}_1\left[R^N_1\right]\right)^2}=0\, \label{largenumberscond2}.	
			\end{align}
			
			Since $\breve F^{2,N}_{m,1}$is stochastically bounded by $F^{N}_{m,1}$, 
			it is enough to show that
			\begin{equation*}
				\lim_{N\to\infty}\frac{1}{b_N}\frac{\mathbb{E}_1\left[\left(F^{N}_{m,1}\right)^2\right]}{\left(\mathbb{E}_1\left[R^N\right]\right)^2}=0\, 
				\Longleftrightarrow
				\lim_{N\to\infty}\frac{1}{c_N}\frac{\mathbb{E}_1\left[\left(F^N_{m,1}\right)^2\right]}{\mathbb{E}_1\left[R^N\right]}=0\, .
			\end{equation*}
			Now, observe that 
			\begin{equation*}
				(F^N_{m,1})^2 = \left(\sum_{\ell=1}^{S}H^N_{m,\ell}\right)^2\le S\sum_{\ell=1}^{S}\left(H^N_{m,\ell}\right)^2,\,
				\text{ where } S=S_{i_m}^1;
			\end{equation*}
			so, conditioning on  $J^{1,N}$, we get
			\begin{equation*}
				\mathbb{E}_1\left[ \left(F^N_{m,1}\right)^2\right]\le \mathbb{E}_{\pi_2}\left[\left(H^N_{m,1}\right)^2\right]\mathbb{E}_{\sigma_1(i_1)}\left[S^2\right]\, .
			\end{equation*}
			As a consequence of Corollary 1.5 in \cite{gayrard2008}, 
			to the effect that the steps of $J^{1,N}$ among the configurations $\sigma_1(i)$, $i\in\M$, is approximately uniformly distributed,
			there exists a positive constant $D$ such that $
			\mathbb{E}_{\sigma_1(i_1)}\left[S^2\right]\le D$ for all $N$; it thus suffices to prove
			\begin{equation}\label{eq_4.1.9}
				\lim_{N\to\infty}\frac{1}{c_N}\frac{\mathbb{E}_{\pi_2}\left[\left(H^N_{m,1}\right)^2\right]}{\mathbb{E}_1\left[R^N_1\right]}=0\, .
			\end{equation}
			Let us write
			\begin{align*}
				&\left(\sum_{j=0}^{g_1-1}e^{\beta\sqrt{N}\Xi_{\sigma_1(i_m)J^{2,N}(j)}}T_j\right)^2
				=\sum_{j=0}^{g_1-1}e^{2\beta\sqrt{N}\Xi_{\sigma_1(i_m)J^{2,N}(j)}}T^2_j \\
				&\hspace{4.6cm}+\, 2\!\!\!\sum_{0\le j<\ell\le g_1-1}e^{\beta\sqrt{N}\Xi_{\sigma_1(i_m)J^{2,N}(j)}}e^{\beta\sqrt{N}\Xi_{\sigma_1(i_m)J^{2,N}(\ell)}}T_jT_\ell\, .
			\end{align*}
			Observe that 
			\begin{equation*}
				\mathbb{E}_{\pi_2}\left[\sum_{j=0}^{g_1-1}e^{2\beta\sqrt{N}\Xi_{\sigma_1(i_m)J^{2,N}(j)}}T^2_j\right]
				=\widetilde{\mathbb{E}}\left[T^2_1\right]\frac{\mu_N(i_m)}{2^{N_2}}\sum_{\sigma_2\in\mathcal{V}_{N_2}}e^{2\beta\sqrt{N}\Xi_{\sigma_1(i_m)\sigma_2}}\, .
			\end{equation*}
			
			Then
			\begin{align*}
				&\hspace{3cm}\frac{N^2}{(N_1\mu_N(i_m))^2}\frac{\mathbb{E}_{\pi_2}
					\left[\sum_{j=0}^{g_1-1}e^{2\beta\sqrt{N}\Xi_{\sigma_1(i_m)J^{2,N}(j)}}T^2_j\right]}{c_N\mathbb{E}_1\left[R^N\right]}&\\
				&=\widetilde{\mathbb{E}}\left[T_1^2\right]\frac{N}{N_2}\frac{\sum_{\sigma_2\in\mathcal{V}_{N_2}}\left(\gamma^N(\sigma_1(i_m)\sigma_2)\right)^2}{\sum_{m=1}^{M}\sum_{\sigma_2\in\mathcal{V}_{N_2}}\left[\gamma^N(\sigma_1(i_m)\sigma_2)+\gamma^N(\sigma_1(j_m)\sigma_2)\right]}\frac{e^{-\beta\sqrt{aN}\left(\left[\Xi_{\sigma_1(i_m)}^{(1)}-\beta_*\sqrt{aN}\right]-L\right)}}{1+\frac{N_1}{N_2}e^{-\beta\sqrt{aN}\Xi^{(1)}_{\sigma_1(i_m)}}}\, .&
			\end{align*}
			By \eqref{convergence_exp}, \eqref{limit_thm_1} and the definition of $i_m$, for which $W_{i_m}>L$, we get that the expression above vanishes as $N\to\infty$.

			For the estimate of the remaining part of $\mathbb{E}_{\pi_2}\left[\left(H^N_{m,1}\right)^2\right]$, observe that
			\begin{align*}
				&\mathbb{E}_{\pi_2}\left[\sum_{0\le j<\ell\le g_1-1}e^{\beta\sqrt{N}\Xi_{\sigma_1(i_m)J^{2,N}(j)}}
				e^{\beta\sqrt{N}\Xi_{\sigma_1(i_m)J^{2,N}(\ell)}}T_jT_\ell\right]\\
				&=\sum_{\sigma_2,\tilde{\sigma_2}\in\mathcal{V}_{N_2}}e^{\beta\sqrt{N}\Xi_{\sigma_1(i_m)\sigma_2}}
				e^{\beta\sqrt{N}\Xi_{\sigma_1(i_m)\tilde{\sigma_2}}}\mathbb{E}_{\pi_2}
				\left[\sum_{0\le j<\ell\le g_1-1}\mathbb{1}_{\{J^{2,N}(j)
					=\sigma_2\}}\mathbb{1}_{\{J^{2,N}(\ell)=\tilde{\sigma_2}\}}\right]\, .
			\end{align*}
			To estimate the second factor in the last line above, we split the summation inside the expectation considering separately values of $j$ and $\ell$ whose difference is smaller or bigger than $N^3$, then using \eqref{def_mu}, \eqref{BG} and some straightforward computations we obtain that there exists a constant $C>0$, independent of $N$, such that, for all $N$ large, holds
			\begin{equation*}
				\mathbb{E}_{\pi_2}\left[\sum_{0\le j<\ell\le g_1-1}\mathbb{1}_{\{J^{2,N}(j)
					=\sigma_2\}}\mathbb{1}_{\{J^{2,N}(\ell)=\tilde{\sigma_2}\}}\right]\leq CN^3\frac{e^{\beta\sqrt{aN}\Xi^{(1)}_{\sigma_1(i_m)}}}{2^{N_2}}\, .
			\end{equation*}
			Then
			\begin{align*}
				&\frac{N^2}{(N_1\mu_N(i_m))^2}\frac{\mathbb{E}_{\pi_2}
					\left[\sum_{0\le j<\ell\le g_1-1}e^{\beta\sqrt{N}\Xi_{\sigma_1(i_m)J^{2,N}(j)}}e^{\beta\sqrt{N}
						\Xi_{\sigma_1(i_m)J^{2,N}(\ell)}}T_jT_\ell\right]}{c_N\mathbb{E}_1[R^N]}\\
				&\leq \frac{C\frac{N}{N_2}}{1+\frac{N_1}{N_2}e^{-\beta\sqrt{aN}\Xi_{\sigma_1(i_m)}^{(1)}}} \frac{\sum_{\sigma_2\in\mathcal{V}_{N_2}}\left(\gamma^N(\sigma_1(i_m)\sigma_2)\right)^2\left[e^{-\beta\sqrt{aN}\left(\left[\Xi_{\sigma_1(i_m)}^{(1)}-\beta_*\sqrt{aN}\right]-L\right)}\right]N^3}{\sum_{m=1}^{M}\sum_{\sigma_2\in\mathcal{V}_{N_2}}\left[\gamma^N(\sigma_1(i_m)\sigma_2)+\gamma^N(\sigma_1(j_m)\sigma_2)\right]}\, ,
			\end{align*}
			which goes to zero as $N\to\infty$, and this establishes~\eqref{largenumberscond2}. The proof of \eqref{eq_4.1.3} is complete.
			
			\bigskip
			
			Let us prove now \eqref{eq_4.1.4}. Notice that $Q^N_{m,k},k\geq 1,$ have the same distribution.
			Now, consider the following indicator random variables:
			\begin{equation*}
				\widehat{B}^N_k:=
				\left\lbrace
				\begin{array}{l}
					1; \text{ if } \sigma(j_m)    \text{ is visited by } J^N \text{ between the } 	k^{th} \text{ and the  } (k+1)^{th} \text{ return of } J^{1,N} \text{ to } \sigma_1(i_1)\, ,\\
					0; \text{ otherwise } \, ,
				\end{array}
				\right. 
			\end{equation*}
			and let us write
			\begin{equation*}
				Q^N_{m,k}=Q^N_{m,k}\widehat{B}^N_k +Q^N_{m,k} (1-\widehat{B}^N_k).
			\end{equation*}	
			
			Proceeding as in the proof of Lemma \ref{lemma_4.1.1}, similarly as in \eqref{expectationF_m}, we find that
			\begin{equation*}
				\mathbb{E}_1\left[Q^{N}_{m,1}(1-\widehat B^N_1)\right]\leq \frac{N}{N_1}\frac{1}{2^{N_2}}\sum_{\sigma_2\in\mathcal{V}_{N_2},\sigma_2\neq\sigma_2(j_m)}e^{\beta\sqrt{N}\Xi_{\sigma_1(j_m)\sigma_2}}\, .
			\end{equation*}	
			Then, using \eqref{expectationRN} and the normalization $u_N^{-1}$, we obtain  
			\begin{equation*}
				\mathbb{E}_1\left[\frac{1}{b_N}\sum_{k=1}^{b_N}\frac{Q^{N}_{m,k}(1-\widehat B^N_k)}{\mathbb{E}_1\left[R^N_1\right]} \right]
				\leq\frac{\sum_{\sigma_2\neq\sigma_2(j_m)}\gamma^N(\sigma_1(j_m)\sigma_2)}{\sum_{m=1}^M\sum_{\sigma_2\in\mathcal{V}_{N_2}}
					\left[\gamma^N(\sigma_1(i_m)\sigma_2)+\gamma^N(\sigma_1(j_m)\sigma_2) \right]}\, .
			\end{equation*}
			By \eqref{convergence_exp}, the right hand side above goes to zero as $N\to\infty$, which proves that
			\begin{equation}\label{eq_4.1.6}
				\frac{1}{b_N}\sum_{k=1}^{b_N} \frac{Q^N_{m,k}(1-\widehat{B}^N_k)}{\mathbb{E}_1[R^N_1]} = 0
			\end{equation} 
			in probability. 
			
			
			
			In order to conclude we will show that 
			\begin{equation} 
				\P_1\big(\cup_{k=1}^{b_N}\widehat{B}^N_k\big)\to0 \text{ as } N\to\infty,
			\end{equation}
			which follows from 
			\begin{equation} 
				b_N\P_1\big(\widehat{B}^N_1\big)\to0 \text{ as } N\to\infty.
			\end{equation}
			
			For $m=1,\ldots, M$ and $\ell\geq1$, let $\hat\tau_\ell(m)$ 
			and $\hg_\ell$ be as in the paragraph of~\eqref{gls}, except that we replace $i_m$ by $j_m$. 
			When $J^{1,N}$ is visiting state $\sigma_1(j_m)$ for the $\ell^{th}$ time, $\ell\geq1$, we consider $\vartheta_\ell$ defined as the number of steps that $J^{2,N}$ takes to reach $\sigma_2(j_m)$, if ever during that visit. We then have that 
			\begin{equation}
				\widehat{B}^N_1=\cup_{\ell=1}^{\hat S}\{\hg_\ell>\vartheta_\ell\},\,\text{ where } \hat S=S^1_{j_m}.
			\end{equation}
			From the elementary theory of Markov chains, it follows that $\E_1(\hat S)=1$; see (7.17) in \cite{fg2018}.
			It is thus enough to show that 
			\begin{equation}\label{est1}
				b_N\P_1\big(\hg_1>\vartheta_1\big)\to0 \text{ as } N\to\infty. 
			\end{equation}
			
			Because of $\hat{g}_1$ is a Geometric random variable with mean $1+\frac{N_2}{N_1}e^{\beta\sqrt{aN}\Xi^{(1)}_{\sigma_1(j_m)}}$, the latter probability is readily seen to equal $\E_1[(1-q)^{\vartheta_1}]$, where $q=\big(1+\frac{N_2}{N_1}e^{\beta\sqrt{aN}\Xi^{(1)}_{\sigma_1(j_m)}}\big)^{-1}$.
			We may then use Kemperman's formula to write 
			\begin{equation}\label{kemp}
				\E_1[(1-q)^{\vartheta_1}]=\E_{\pi_2}[(1-q)^{\vartheta_1}]=\frac1{B_0(\l)}\frac1{2^{N_2}}\sum_{i=0}^{N_2}{N_2\choose i} B_i(\l),
			\end{equation}
			where for $i=0,\ldots,N_2$, 
			\begin{equation}\label{Bi}
				B_i(\l)=\int_0^1(1-u)^i(1+u)^{N_2-i}u^{\l-1}du=\sum_{j=0}^{N_2-i}{N-i\choose j}\frac{\G(i+1)\G(\l+j)}{\G(\l+i+j+1)},
			\end{equation}
			with 
			$\l=\l_N=\frac{N_2}2\frac{q}{1-q}=\frac{N_1}{2}e^{-\beta\sqrt{aN}\Xi^{(1)}_{\sigma_1(j_m)}}
			=\frac{N_1}{2}e^{-\beta\beta_*aN}e^{-\beta\sqrt{aN}(W_{j_m}+o_1)}$; 
			see (4.13-14) in~\cite{fg2018}.
			We readily find that the expression to the left in~\eqref{est1} equals
			\begin{align}\label{est2}
				\frac{b_N}{B_0(\l)}\int_0^1u^{\l-1}du\frac1{2^{N_2}}\sum_{i=0}^{N_2}{N_2\choose i} (1-u)^i(1+u)^{N_2-i}
				&=\frac{b_N}{B_0(\l)}\int_0^1u^{\l-1}du=\frac{b_N}{\l B_0(\l)}\nonumber\\
				&=\frac{b_N}{1+\l\sum_{i=1}^{N_2}{N_2\choose i}\frac{1}{i+\l}},
			\end{align}
			and that the sum in the denominator on the right hand side above is 
			\begin{equation}\label{kempf}
				\sim\sum_{i=1}^{N_2}{N_2\choose i}\frac1{i}\sim\frac{2^{N_2+1}}{N_2}.
			\end{equation}
			
			From the above 
			and~\eqref{b_N}, we find that the left hand side of~\eqref{est1} equals
			\begin{align}
				\frac{N_1}{2N}
				\frac{e^{\beta\left(\beta_*N- \frac{\log N +\kappa}{2\beta_*}\right)}}{
					\sum_{m=1}^{M}\left[\sum_{\sigma_2\in\mathcal{V}_{N_2}}e^{\beta\sqrt{N}\Xi_{\sigma_1(i_m)\sigma_2}}	+\sum_{\sigma_2\in\mathcal{V}_{N_2}}e^{\beta\sqrt{N}\Xi_{\sigma_1(j_m)\sigma_2}}\right]}\nonumber\\
				\times\frac{e^{-\beta\beta_*aN-\beta\sqrt{aN}L}}{2^{-N_2}+\frac{N_1}{N_2}e^{-\beta\beta_*aN}e^{-\beta\sqrt{aN}(W_{j_m}+o_1)}}.
			\end{align}
			The first quotient above is clearly of order 1, and the second quotient may be also checked to be of order 1 (since the numerator is the right scale for the denominator, as follows from~\eqref{BK} and~\eqref{def2}), and we readily check that the third quotient
			is of order $e^{-\beta\sqrt{aN}(L-W_{j_m}+o_1)}$ (here we have used that $\beta<\bar\beta_{FT}$), and this vanishes as $N\to\infty$ since by definition $W_{j_m}<L$. The result follows.

			
		\end{proof}
		

		\begin{corollary}\label{large_numbers_coro}  
			Given $\delta>0$, consider  $b_N:=\left\lfloor  \frac{\delta c_N}{\mathbb{E}_1[R^N_1]} \right\rfloor$. Then for all $m=1,\dots, M$ we have
			\begin{equation*}
				\lim_{N\to\infty}\frac{1}{b_N}\sum_{k=1}^{b_N} \frac{R^N_k}{\mathbb{E}_1\left[R^N_1\right]} = \frac{\sum_{m=1}^{M}\gamma(i_m)}{\sum_{m=1}^{M}\left[ \gamma(i_m)+\gamma(j_m)\right]}\, ,
			\end{equation*}
			in probability.
		\end{corollary}	
		
		\begin{proof}
			Follows immediately from Lemma \ref{lemma_4.1.2}.
		\end{proof}
		
		\begin{lemma}\label{lemma_L^N} Let 
			\begin{equation}\label{L^N}
				L^N(t):=\max\left\{ n\geq 0: \sum_{k=0}^n R^N_k \le t\right\}\, , \, \text{ for } t\ge 0.
			\end{equation}
			Then
			\begin{equation*}
				\lim_{N\to\infty} \frac{\mathbb{E}_1\left[R^N_1\right]L^N(c_N)}{c_N} = \frac{\sum_{m=1}^{M}\left[\gamma(i_m)+\gamma(j_m)\right]}{\sum_{m=1}^{M}\gamma(i_m)}\, ,
			\end{equation*}
			in probability.
		\end{lemma}	
		
		\begin{proof}
			Follows from Corollary \ref{large_numbers_coro}. However, since the summation in the definition of $L^N$ starts at $k=0$, we only need to prove that the first summand, when divided by $c_N$, goes to zero in probability. Indeed, using that $
			\mathbb{E}_{\pi_1}\left[\sum_{k=0}^{\tau^1-1}\mathbb{1}_{\{J^{1,N}(k)=\sigma_1\}}\right]\leq 2
			$ for any $\sigma_1\in\mathcal{V}_n$ and $\tau^1$ as defined in \eqref{tau}, (see Lemma 7.4 in \cite{fg2018}), and reproducing the estimate \eqref{expectationRN}, we get
			\begin{equation*}
				\mathbb{E}_{\pi_1\times\pi_2}[R_0^N]\leq 2 \frac{N}{N_1}\frac{1}{2^{N_2}}\sum_{m=1}^{M}\left[\sum_{\sigma_2\in\mathcal{V}_{N_2}}e^{\beta\sqrt{N}\Xi_{\sigma_1(i_m)\sigma_2}}	+\sum_{\sigma_2\in\mathcal{V}_{N_2}}e^{\beta\sqrt{N}\Xi_{\sigma_1(j_m)\sigma_2}}\right]\, .
			\end{equation*}
			Since $\beta<\beta_{FT}$ we obtain the desired convergence.
		\end{proof}	
		
		\begin{lemma} \label{renewal_lemma}For $m=1,\dots, M$ and $L^N$ as defined in \eqref{L^N}, we have
			\begin{equation*}
				\lim_{N\to\infty}\frac{1}{L^N(c_N)}\sum_{k=0}^{L^N(c_N)}\frac{F^N_{m,k}}{\mathbb{E}_1\left[F^N_{m,k}\right]} = 1\, ,
			\end{equation*}
			in probability.
		\end{lemma}	
		\begin{proof}
			Follows readily from Lemmas \ref{lemma_4.1.1} and \ref{lemma_L^N}.
		\end{proof}	
		
		\begin{proof}[Proof of Proposition \ref{prop_4.1}] 
			Note that
			\begin{equation}\label{eq_4.1.20}
				\sum_{k=0}^{L^N(c_N)}F^N_{\ell,k}\le \int_{0}^{c_N}\mathbb{1}_{\{X^N_M(s)=\ell\}}ds\le \sum_{k=0}^{L^N(c_N)+1}F^N_{\ell,k}\, .
			\end{equation}
			Writing 
			\begin{align*}
				\frac{1}{c_N}\sum_{k=0}^{L^N(c_N)}F^N_{\ell,k}=\left(\frac{1}{L^N(c_N)}\sum_{k=0}^{L^N(c_N)}\frac{F^N_{\ell,k}}{\mathbb{E}_1\left[F^N_{\ell,1}\right]} \right)\frac{\mathbb{E}_1\left[F^N_{\ell,1}\right]}{\mathbb{E}_1\left[R^N_1\right]}\left(\frac{L^N(c_N)\mathbb{E}_1\left[R^N_1\right]}{c_N}\right)
			\end{align*}
			and using Lemma \ref{lemma_4.1.1}, Lemma \ref{lemma_L^N} and Lemma \ref{renewal_lemma} we obtain the desired convergence.
		\end{proof}

		\subsection{Proof of Proposition \ref{prop_4.2}}\label{sub_seq_4.2}
		
		Let $O_0^N$ be the time spent by $X^N$ outside $I_M\cup J_M$ until its first visit to $i_1$, and for $k\ge 1$, let  $O_k^N$ be the time spent outside of $I_M\cup J_M$ by $X^N$ between the $k^{th}$ and the $(k+1)^{th}$ visit to $i_1$. Reproducing the estimate \eqref{expectationF_m} and using that $\mathbb{E}_{\pi_1}\left[\sum_{k=0}^{\tau-1}\mathbb{1}_{\{J^{1,N}(k)=\sigma_1\}}\right]\leq 2
		$ (see Lemma 7.4 in \cite{fg2018}), we get
		\begin{align}\label{eq_4.2.1}
			\mathbb{E}_{\pi_1\times\pi_2}\left[O_0^N\right]&\leq  2\sum_{i\notin I_M\cup J_M}\frac{1}{2^{N_2}}\sum_{\sigma_2\in\mathcal{V}_{N_2}}\frac{N}{N_1}e^{\beta\sqrt{N}\Xi_{\sigma_1(i)\sigma_2}} \, ,
		\end{align}
		and for $k\ge 1$,
		\begin{align}\label{Uk}
			\mathbb{E}_1\left[O^N_k\right]&=\sum_{i\notin I_M\cup J_M}\frac{1}{2^{N_2}}\sum_{\sigma_2\in\mathcal{V}_{N_2}}\frac{N}{N_1}e^{\beta\sqrt{N}\Xi_{\sigma_1(i)\sigma_2}} \, .
		\end{align}	
		Notice that 
		\begin{equation*}
			T^{N,out}_M(1)\leq O_0^N + \sum_{k=1}^{L^N(c_N)+1}O_k^N\, .
		\end{equation*}	
		For any $\lambda>0$ we have
		\begin{equation*}
			\mathbb{P}_{\pi_1\times\pi_2}\left[\frac{1}{c_N}T^{N,out}_M(1)>\lambda\right]\leq \frac{2}{\lambda}\frac{\mathbb{E}_{\pi_1\times\pi_2}\left[O_0^N\right]}{c_N} + \mathbb{P}_{\pi_1\times\pi_2}\left[\frac{1}{c_N}\sum_{k=1}^{L^N(c_N)+1}O^N_k>\frac{\lambda}{2}\right].
		\end{equation*}
		From \eqref{eq_4.2.1} we get that
		\begin{equation*}
			\limsup_{N\to\infty}\frac{\mathbb{E}_{\pi_1\times\pi_2}\left[O_0^N\right]}{c_N}  \leq  2\limsup_{N\to\infty}\frac{1}{2^{N_2}}e^{\beta(\beta_*aN+\sqrt{aN}L)}\frac{N}{N_1}\sum_{ i\notin I_M\cup J_M}\sum_{\sigma_2\in\mathcal{V}_{N_2}}\gamma^N(\sigma_1(i)\sigma_2)\, .
		\end{equation*}
		Since, for $\beta<\beta_{FT}$, we have that $\frac{1}{2^{N_2}}e^{\beta(\beta_*aN+\sqrt{aN}L)}$ goes to zero as $N\to \infty$, we obtain 
		\begin{equation}\label{eq_4.2.2}
			\lim_{N\to\infty}\frac{\mathbb{E}_{\pi_1\times\pi_2}\left[O_0^N\right]}{c_N}= 0\, .
		\end{equation}
		To make the notation easier let us denote  $\displaystyle s_M:=\frac{\sum_{m=1}^{M}\left[\gamma(i_m)+\gamma(j_m)\right]}{\sum_{m=1}^{M}\gamma(i_m)} $. Let us write
		\begin{align*}
			&\mathbb{P}_{\pi_1\times\pi_2}\left[\frac{1}{c_N}\sum_{k=1}^{L^N(c_N)+1}O^N_k>\frac{\lambda}{2}\right]\\
			&\leq\mathbb{P}_{\pi_1\times\pi_2}\left[\frac{1}{c_N}\sum_{k=1}^{L^N(c_N)+1}O^N_k>\frac{\lambda}{2}; \frac{L^N(c_N)\mathbb{E}_1\left[R^N_1\right]}{c_N}\leq  2 s_M\right] + \mathbb{P}_{\pi_1\times\pi_2}\left[\frac{L^N(c_N)\mathbb{E}_1\left[R^N_1\right]}{c_N} > 2 s_M \right]\, .
		\end{align*}
		By Lemma \ref{lemma_L^N}, the last term above goes to zero as $N\to\infty$. The summation inside the probability of the other term can be bounded by the summation up to $\displaystyle k=\frac{2s_Mc_N}{L^N(c_N)\mathbb{E}_1[R_1^N]}+1$, then the first term in the last line above, by Markov's inequality, can be bounded above by
		\begin{equation*}
			\frac{2}{\lambda c_N}\left[\frac{2s_Mc_N}{\mathbb{E}_1\left[R^N_1\right]}+1\right]\mathbb{E}_1\left[O^N_1\right]\, .
		\end{equation*}
		By \eqref{convergence_exp}, \eqref{expectationRN} and \eqref{Uk}, we have
		\begin{equation*}
			\lim_{N\to\infty}\frac{\mathbb{E}_1\left[O^N_1\right]}{\mathbb{E}_1\left[R^N_1\right]}
			=\frac{\sum_{i\notin I_M\cup J_M }\gamma(i)}{\sum_{m=1}^{M}\left[\gamma(i_m)+\gamma(j_m)\right]}\, .
		\end{equation*}
		Analogously to \eqref{eq_4.2.2}, we can get 
		\begin{equation*}
			\lim_{N\to\infty}\frac{\mathbb{E}_1\left[O^N_1\right]}{c_N} = 0\, .
		\end{equation*}
		Then
		\begin{equation*}
			\lim_{N\to\infty}\frac{1}{c_N}\left[\frac{2s_Mc_N}{\mathbb{E}_1\left[R^N_1\right]}+1\right]\mathbb{E}_1\left[O^N_1\right] =2s_M\frac{\sum_{i\notin I_M\cup J_M}\gamma(i)}{\sum_{m=1}^{M}\gamma(i_m)}\, .
		\end{equation*}
		The last factor in  the right hand side above goes to zero as $M\to\infty$. Hence the proof of Proposition \ref{prop_4.2} is complete.

		\section{Below fine tuning}
		\setcounter{equation}{0}
		
		The behavior in this regime is quite similar to the one for the corresponding regime in the cascading case analysed in~\cite{fg2018},
		and so is the proof of Theorem~\ref{Thm_below_FT}, when compared to Theorem 2.7 of~\cite{fg2018}; so, we will be quite concise in this section, arguing some proofs explicitly, for a measure of completeness, and pointing to a result of~\cite{fg2018} in a particular  instance.
		
		We recall that we are in the framework of Skorohod Representation for the environment, as pointed out in more detail in~(\ref{limit_thm_1}) above.
		
		Given a fixed $M>0$ let us consider $\mathcal{M}:=\{ 1,\dots,M \}$.	Since the motion between first level low-lying configurations is approximately uniform (as follows from Corollary 1.5 in~\cite{gayrard2008}), it is enough to establish two things in the time scale considered: that the 
		time spent in each visit to $\mathcal{M}$ is approximately exponential, with the appropriate rate, 
		and that the time spent between successive visits
		to  $\mathcal{M}$ is negligible. We will devote a subsection for each of these issues.


		\subsection{Time spent in each visit  to $\mathcal{M}$.}
		
		In this subsection we will  study the time spent by $X^N$ in each visit to one of the states in the set 
		$\mathcal{M}=\{1,\cdots, M\}$. It is enough to consider the first visit --- see Remark~\ref{1st_vis} below.
		\begin{proposition}\label{time_spent_in_each_visit_to_$M_1$_above_FT} 
			For each $i\in\mathcal{M}$, let $\Psi^N_i$ be the time spent by $X^N$ during its first visit to $\{i\}$. 
			Then, for $\bar c_N$ defined in \eqref{bar_c_N}, we have that $\bar c^{-1}_N\Psi^N_i$ converges in distribution to a mean $e^{\frac{\beta}{\beta_*}\xi_i}$  exponential random variable, as $N$ goes to infinity.
		\end{proposition}	
		\begin{proof}
			Let us denote by $\Upsilon^N_i$ and  $\Gamma^N_i$  the time spent by $\sigma^N$ on $\sigma^N(i)=\sigma_1^N(i)\sigma_{2}^N(i)$, respectively  outside of $\sigma^N(i)$, during the first visit of $X^N$ to $\{i\}$. Note that
			\begin{equation}\label{Psi}
				\Psi^N_i = \Upsilon^N_i + \Gamma^N_i.
			\end{equation}
			
			We will prove Proposition \ref{time_spent_in_each_visit_to_$M_1$_above_FT} in two steps.
			
			\underline{Step one:} Let us show  that $\bar c^{-1}_N\Upsilon_i^N$ converges in distribution to a mean $e^{\frac{\beta}{\beta_*}\xi_i}$ exponential random variable.
			
			To simplify the notation let us write $\sigma^N(i)$ as $\sigma = \sigma_1\sigma_2$. At each time the process reaches $\sigma_1$ it executes a geometric numbers of steps before leaving $\sigma_1$. Let us call this geometric random variable by ${G}^N$. Notice that ${G}^N$ has mean $1+\frac{N_2}{N_1}e^{\beta\sqrt{aN}\Xi^{(1)}_{\sigma_1}}$. At each time the process $\sigma^N$ arrives at $\sigma$ it spends an exponential time of mean $\frac{N_2}{N}e^{\beta\sqrt{(1-a)N}\Xi^{(2)}_{\sigma}}$. 
			Hence  taking, as before, $\{T_j,\, j\geq 0\}$ an  i.i.d.~family of mean one exponential random variables, 
			and $\{J^{2,N}(j),\,j\ge 0\}$ the discrete random walk on $\mathcal{V}_{N_2}$ associated to $\sigma_N$ we get,
			\begin{align*}\label{Upsilon}
				\Upsilon_i^N &=\frac{N_2}{N}e^{\beta\sqrt{(1-a)N}\Xi_{\sigma}^{(2)}}\sum_{j=0}^{{G}^N-1}1_{\{ J^N_2(j)=\sigma_2\}}T_j\\
				&= \frac{(N_2)^2}{N_1N}e^{\beta\sqrt{N}\Xi_{\sigma}}2^{-N_2}\left(\frac{N_1}{N_2}
				e^{-\beta\sqrt{aN}\Xi^{(1)}_{\sigma_1}}\right){G}^N\,\frac{2^{N_2}}{{G}^N}\sum_{j=0}^{{G}^N}1_{\{J^{2,N}(j)=\sigma_2\}}T_j.
			\end{align*}
			Then
			\begin{equation*}
				\bar c^{-1}_N\Upsilon_i^N  = e^{-\beta\beta_*N 
					+\frac{\beta}{2\beta_*}(\log N + \kappa)} e^{\beta\sqrt{N}\Xi_{\sigma}}
				\left(\frac{N_1}{N_2}e^{-\beta\sqrt{aN}\Xi^{(1)}_{\sigma_1}}\right)
				{G}^N\,\frac{2^{N_2}}{{G}^N}\sum_{j=0}^{{G}^N}1_{\{J^{2,N}(j)=\sigma_2\}}T_j.
			\end{equation*}
			From our assumptions on the environment, we have that a.s.~$e^{(\beta/\beta_*)u_N^{-1}(X_{\sigma})}\to e^{(\beta/\beta_*)\xi_i}$ as $N$ goes to infinity, and it follows readily that 
			$$
			\left(\frac{N_1}{N_2}e^{-\beta\sqrt{aN}\Xi^{(1)}_{\sigma_1}}\right){G}^N
			$$
			converges to a mean one exponential random variable 
			So it is enough to prove that 
			\begin{equation}\label{below_eq1}
				\frac{2^{N_2}}{{G}^N}\sum_{j=0}^{{G}^N-1}\mathbb{1}_{\{J^N_2(j)=\sigma_2\}}T_j
			\end{equation}
			converges to one in probability as $N$ goes to infinity. But this object has already been considered in 
			the proof of Lemma 7.3 of~\cite{fg2018}, where this result was established; see (7.4) and (7.6) in~\cite{fg2018}.

			\underline{Step two:} Let us show  that $\bar c^{-1}_N\Gamma_i^N$ goes to zero in probability.
			
			The variable $\Gamma_i^N$ can be expressed as 
			\begin{align*}
				\Gamma^N_i &= \sum_{j=0}^{{G}^N-1}\mathbb{1}_{\{J^{2,N}(j)\neq\sigma_2\}} e^{\beta\sqrt{(1-a)N}\Xi^{(2)}_{\sigma_1J^{2,N}(j)}}T_j
				=\sum_{\sigma_2'\neq\sigma_2}e^{\beta\sqrt{(1-a)N}\Xi^{(2)}_{\sigma_1\sigma_2'}}
				\sum_{j=0}^{{G}^N-1}\mathbb{1}_{\{J^N_2(j) = \sigma_2'\}}T_j.
			\end{align*}
			{Recall that initial distribution of $J^{2,N}$ is the uniform distribution on the hypercube $\{-1,1\}^{N_2}$}, denoted by $\pi_2$, which is invariant for $J^{2,N}$. Then
			\begin{align*}
				\mathbb{E}_{\pi_2}[\Gamma^N_i ] =\frac{N_2}{N_1} \sum_{\sigma_2'\neq\sigma_2}e^{\beta\sqrt{(1-a)N}\Xi^{(2)}_{\sigma_1\sigma_2'}}e^{\beta\sqrt{aN}\Xi^{(1)}_{\sigma_1}}\frac{1}{2^{N_2}}= \frac{N_2}{N_12^{N_2}} \sum_{\sigma_2'\neq\sigma_2}e^{ \beta\sqrt{N}\Xi_{\sigma_1\sigma_2 '}}.
			\end{align*}
			Then 
			$$\bar c^{-1}_N\mathbb{E}_{\pi_2}[\Gamma^N_i]
			=\frac{N}{N_2}\sum_{\sigma'_2\neq\sigma_2}e^{\frac{\beta}{\beta_*}u_N^{-1}(\Xi_{\sigma_1\sigma'_2})},$$
			which may be readily checked from our assumptions and properties of our environment to a.s.~converge to zero.
			
		\end{proof}
		
		\begin{remark}\label{1st_vis}
			The successive visit durations to a given state of $\M$ are identically distributed, due to stationary of $J^{2,N}$
			under the uniform initial distribution. It is quite clear from the above proof, in particular from the convergence  
			of~(\ref{below_eq1}), that the limiting exponential random variables are independent.
		\end{remark}

		\subsection{Time spent outside $\mathcal{M}$.} \label{out}
		As in the previous section, consider $i\in\mathcal{M}=\{1,\dots,M\}$ and  denote by $U^N_i$ the time spent by $X^N$ outside $\mathcal{M}$ until the first visit to $\{i\}$. Then
		\begin{align*}
			U^N_i &= \sum_{k=0}^{\tau_{i}^N}\mathbb{1}_{\{J^{1,N}(k)\notin \mathcal{M}\}}
			\sum_{j=0}^{{G}_{k}^N(J^{1,N}(k))-1}e^{\beta\sqrt{(1-a)N}\Xi^{(2)}_{J^{1,N}(k)J^{2,N}({\mathfrak G}_{k-1}^N+j)}}T_j^k\\
			&=\sum_{\ell \notin \mathcal{M} }\sum_{\sigma_2\in \mathcal{V}_{N_2}}e^{\beta\sqrt{(1-a)N}\Xi^{(2)}_{\sigma_1'\sigma_2}}
			\sum_{k=0}^{\tau_i^N  }\mathbb{1}_{\{J^{1,N}(k)=\sigma_1(\ell)\}}
			\sum_{j=0}^{{G}_{k}^N(\ell)-1}\mathbb{1}_{\{J^{2,N}({\mathfrak G}_{k-1}^N+j)=\sigma_2\}}T_j^k,
		\end{align*}
		where $\{J^{1,N}(j),\,j \geq 1\}$ is the discrete random walk on $\mathcal{V}_{N_1}$ associated to $\sigma^N$; 
		${\mathfrak G}_{k}^N = \sum_{n=0}^{k}{G}_n^N(J^{1,N}(n))$, $k\geq0$, ${\mathfrak G}_{-1}^N =0$,
		and $\tau_{i}^N$ is the first time that $J^{1,N}$ reaches $\sigma_1(i)$. 
		As in Section 7.3.2 of \cite{fg2018}, we have 
		\begin{equation*}
			\mathbb{E_{\pi_1\times\pi_2}}[U_i^N] = \sum_{\ell\notin\mathcal{M}}
			\left(\sum_{\sigma_2\in\mathcal{V}_{N_2}}e^{\beta\sqrt{N}\Xi_{\sigma_1(\ell)\sigma_2}}
			\frac{1}{2^{N_2}}\mathbb{E}_{\pi_2}\left[\sum_{k=0}^{\tau_{i}^N}\mathbb{1}_{\{J^{1,N}(k)=\sigma_1(\ell)\}}\right]\right).
		\end{equation*}	
		By Lemma 7.4 of \cite{fg2018}, for $N$ large enough, we have 
		\begin{equation}\label{U_N}
			\mathbb{E_{\pi_1\times\pi_2}}[U^N] \leq 2\sum_{\ell\notin \mathcal{M}}
			\left(\sum_{\sigma_2\in\mathcal{V}_{N_2}}e^{\beta\sqrt{N}\Xi_{\sigma_1(\ell)\sigma_2}}\frac{1}{2^{N_2}}\right).
		\end{equation}	
		
		Let now $\mathcal{W}_{i,j}^N$  be the time spent by $X^N$ outside $\mathcal{M}$ between the $j^{th}$ and $(j+1)^{th}$ visit to $i\in \mathcal{M}$, for $j \geq 1$. Again, as in Section 7.3.2 of \cite{fg2018}, we get	
		\begin{equation*}
			\mathbb{E}_{\sigma(i)}[\mathcal{W}_{i,j}^{N}] 
			= \sum_{\ell\notin \mathcal{M}} \left(\sum_{\sigma_2\in\mathcal{V}_{N_2}}e^{\beta\sqrt{N}\Xi_{\sigma_1(\ell)\sigma_{2}}}
			\frac{1}{2^{N_2}} \mathbb{E}_{\sigma_1}\left[\sum_{k=0}^{\tau_{i}^N}\mathbb{1}_{\{J^{1,N}(k)=\sigma_1(\ell)\}}\right] \right) .
		\end{equation*}
		As pointed out above,
		the expectation in the right hand side above is equal to one. 
		Then
		\begin{equation}\label{W_N}
			\mathbb{E}_{\sigma(i)}[\mathcal{W}_{i,j}^{N}] 
			= \sum_{\ell\notin \mathcal{M}}
			\left(\sum_{\sigma_2\in\mathcal{V}_{N_2}}e^{\beta\sqrt{N}\Xi_{\sigma_1(\ell)\sigma_2}}\frac{1}{2^{N_2}}\right).
		\end{equation}
		Hence we have that $\mathbb{E}_{\pi_1\times\pi_2}[\bar c^{-1}_NU_i^N]$ and  $\mathbb{E}_{\sigma(i)}[\bar c^{-1}_N\mathcal{W}_{i,j}^N]$ are all bounded above by
		$$2\sum_{\ell \notin \mathcal{M}}\sum_{\sigma_2\in\mathcal{V}_{N_2}}e^{\frac{\beta}{\beta_*}u_N^{-1}(\Xi_{\sigma_1(\ell)\sigma_2})},$$
		which, by Theorem \ref{BK}, converges to
		\begin{equation}\label{calda}
			2\sum_{j=M+1}^{\infty}e^{\frac{\beta}{\beta_*}\xi_{j}},
		\end{equation}
		where $\{\xi_j,\,j\geq 1\}$ is a Poisson point process as in Definition \ref{PPP}. 
		It may be readily checked that the latter series is a.s.~convergent, and the result follows.

		\section{At fine tuning}
		\setcounter{equation}{0}
		
		This case is quite similar to the one of the previous section, and again we will be quite concise. Again, since 
		the motion between first level low-lying configurations is approximately uniform, in order to prove Theorem~\ref{Thm_at_FT},
		it is enough to show that the rescaled
		time spent in each visit to any state in $\M_L=\mathcal{M}\cap\N_L$ is approximately exponential, with the appropriate rate, 
		and that the rescaled time spent between successive visits
		to  $\M_L$ is appropriately negligible. We do that in the following two subsections.

		\subsection{Time spent in each visit to $\M_L$}
		
		Now we will study the time spent by the process $X^N$ during visits to a given low-lying configuration $i\in\M$. The next proposition states our result.  For similar reasons as for the result of the previous section, it is enough to consider the first such visit.
		Recall $\Psi_i^N$ defined in the statement of Proposition~\ref{time_spent_in_each_visit_to_$M_1$_above_FT}. 
		\begin{proposition}\label{time_spent_in_each_visit_to_$M_1$_at_FT}
			Let $i\in\M$ be (the label of) a low-lying configuration. Then for all $L\in\R$
			\begin{enumerate}
				\item[(a)] if $W_{i}>L$, then $c_N^{-1}\Psi_i^N$ converges in distribution to a mean $e^{\frac{\beta_{FT}}{\beta_*}\xi_{i}}$ exponential random variable,
				\item[(b)] if $W_{i}< L$, then $c_N^{-1}\Psi_i^N$ converges in probability to zero. 
			\end{enumerate} 
		\end{proposition} 
		Recall $\Upsilon_i^N$ defined in the paragraph of  \eqref{Psi}.
		\begin{lemma} \label{conv_at_FT} 
			Under the same conditions, we have
			\begin{enumerate}
				\item [(a)] if $W_{i}>L$, then $c_N^{-1}\Upsilon_i^N$ converges in distribution to an exponential random variable with mean $e^{\frac{\beta_{FT}}{\beta_*}\xi_{\sigma}}$ as $N$ goes to infinity,
				\item [(b)] if $W_i< L$, then $\Upsilon_i^N=0$ with probability tending to 1 as $N$ goes to infinity.
			\end{enumerate}	
		\end{lemma}	
		\begin{proof}
			In the case $W_i>L$, we may readily check that $\lim_{N\to\infty}\frac{\mathbb{E}[{G}^N]}{2^{N_2}}=+\infty$,
			where ${G}^N$ is as in Step 1 of the proof of Proposition~\ref{time_spent_in_each_visit_to_$M_1$_above_FT}.
			This was the key behind the proof of the convergence results in the {\em above fine tuning} regime. 
			With this fact in hands, {\em (a)} follows in the same way as in the proof of the first step of 
			Proposition \ref{time_spent_in_each_visit_to_$M_1$_above_FT}.
			
			For {\em (b)}, we have
			\begin{equation}\label{prob}
				\mathbb{P}[\Upsilon_i^N>0]
				=\mathbb{P}[{G}^N\geq \tt] =\E[(1-q)^{\tt}],
			\end{equation}
			
			where $\tt$ is the hitting time of $\sigma_2$ by $J^{2,N}$.
			
			Again, by the application of Kemperman's formula to the latter expectation, as explained in the paragraph of~(\ref{kemp}-\ref{kempf}) above, we have that the right hand side of~\eqref{prob} is $\sim\big(1+\frac{2^{N_2+1}}{N_2}\l\big)^{-1}$,
			where $\l=\frac{N_2}2\frac{q}{1-q}=\frac{N_1}{2}e^{-\beta\sqrt{aN}\Xi^{(1)}_{\sigma_1}}$. 
			In the present case $\l=O\big(N_12^{-N_2}e^{\nu(L-W_i+o_1)\sqrt N}\big)$, with $\nu=\frac{1-p}{2\sqrt a}\beta_\ast$, so
			we find that~\eqref{prob} is an $O\big(e^{-\nu(L-W_i+o_1)\sqrt N}\big)$, which in turn vanishes as $N\to\infty$ from our hypothesis.
		\end{proof}	
		\begin{proof}[Proof of Proposition \ref{time_spent_in_each_visit_to_$M_1$_at_FT}]
			Recall from \eqref{Psi} the definition of $\Gamma_i^N$. Using Lemma \ref{conv_at_FT}, we only need to show that $c_N^{-1}\Gamma_i^N$ converges to zero in probability. Proceeding as in the end of the proof of Proposition \ref{time_spent_in_each_visit_to_$M_1$_above_FT}, we obtain the desired result.   	
		\end{proof}	
		
		\subsection{Time spent outside $\M$}
		
		The argument from Subsection~\ref{out} works here as well up to~\eqref{W_N}. Then we must replace $\bar c_N$ by $c_N$,
		and the argument goes over in the same way, except that we get an extra factor of $e^{\frac{1-p}{2a}L^2}$ in~\eqref{calda}, 
		where $\beta$ should be taken as $\bar\beta_{FT}$.

		\section*{Acknowledgements}
		
		LRF was partially supported by CNPq grants 311257/2014-3 and 307884/2019-8; and FAPESP grant 2017/10555-0.
		SF was partially supported by FAPESP grant 2017/10555-0 during visits to USP to carry out this research.
		LZ was supported by a PNPD/CAPES grant 88882.315481/2013-01 post doctoral fellowship. We thank an anonymous referee for many comments to an earlier version of this text that much helped improve the presentation.

		\markboth{References}{References}
		\bibliographystyle{abbrv}
		\bibliography{GREM}

\end{document}